\documentclass[preprint,5p,twocolumn]{elsarticle}

\usepackage{amsmath,amssymb,amsfonts}
\usepackage{amsthm}
\usepackage{graphicx}
\graphicspath{{Figures/}}
\usepackage{etoolbox}
\usepackage{mathrsfs}
\usepackage{algorithm}
\usepackage{algpseudocode}
\usepackage{placeins}
\usepackage{hyperref}
\hypersetup{hidelinks}
\biboptions{sort&compress}
\algrenewcommand\algorithmicrequire{\textbf{Input:}}
\algrenewcommand\algorithmicensure{\textbf{Output:}}

\newtheorem{theorem}{Theorem}
\newtheorem{proposition}{Proposition}
\newtheorem{definition}{Definition}
\newtheorem{corollary}{Corollary}

\newtheorem{remark}{Remark}

\newcommand{\mb}[1]{\mathbb{#1}}
\newcommand{\mc}[1]{\mathcal{#1}}

\newcommand{\N}{\mb{N}}
\newcommand{\R}{\mb{R}}

\newcommand{\support}[2]{\operatorname{h}(#1,#2)}
\newcommand{\supportbb}[1]{\operatorname{h}(#1,\cdot)}

\numberwithin{equation}{section}

\begin{document}
\raggedbottom
\begin{frontmatter}
\title{The Maximal Robust Positively Invariant Set for Linear Difference Inclusions}
\author[liu]{Kai Wang}
\ead{kai.wang@liu.se}
\author[liu]{Daniel Axehill}
\ead{daniel.axehill@liu.se}
\address[liu]{Division of Automatic Control, Department of Electrical
Engineering, Link\"oping University, Sweden}

\begin{abstract}
This article develops a systematic characterization and exact computational framework for the maximal robust positively invariant set of linear difference inclusions. The maximal set is characterized as the greatest fixed point of the robust predecessor operator, leading to three equivalent set iterations for its computation. Under uniform exponential stability and bounded disturbances, the Kuratowski upper limit of the disturbance-reachable sets yields a necessary and sufficient condition for the maximal set to be nonempty under closed state constraints. Sufficient conditions for its finite determination are also established, together with an explicit bound on the determination index. For polyhedral state constraints, bounded disturbances, and finite or polytopic matrix families, exact half-space implementations of the three iterations are derived. Although these algorithms generate the same set sequence, complexity analysis and numerical experiments reveal substantial differences in their computational performance.
\end{abstract}

\begin{keyword}
Set invariance \sep robust positively invariant sets \sep linear
difference inclusions \sep constrained dynamics
\end{keyword}
\end{frontmatter}

\section{Introduction}
\label{sec:introduction}

\subsection{Background and Related Work}
Set invariance provides a geometric framework for certifying persistent constraint satisfaction 
and analyzing the qualitative behavior of dynamical systems. It plays a fundamental role
in stability and region-of-attraction analysis, as well as in set-theoretic controller synthesis
\cite{Khalil2002NonlinearSystems,Blanchini1999SetInvariance,BlanchiniMiani2015}. It is also closely connected to viability theory and reachability analysis \cite{Aubin1991ViabilityTheory,
BertsekasRhodes1971MinimaxReachability, MitchellBayenTomlin2005HJReachability,
KurzhanskiyVaraiya2007EllipsoidalReachability}. An invariant set identifies a region that is preserved under the system evolution. When contained in the prescribed state constraint set, it provides a natural certificate of safety, thereby supporting formal verification
\cite{Tabuada2009VerificationHybridSystems, BeltaYordanovGol2017FormalMethods}.

The natural question of how large such an invariant region can be leads to the study of maximal invariant sets. A broad literature has investigated maximal invariant regions for autonomous and controlled systems, including robust extensions that account for model uncertainty and
external disturbances. This body of work spans theoretical
characterizations, exact computational methods, and techniques for
constructing approximations of these maximal sets
\cite{Bertsekas1972InfiniteReachability,
GutmanCwikel1987MaximalConstraints,
GilbertTan1991MaximalOutputAdmissible,
KolmanovskyGilbert1998DisturbanceInvariant,
DoreaHennet1999ABInvariant,
LinAntsaklis2002RobustControlledInvariant,
Shang2004MaximalRobustControlledInvariant,
PluymersEtAl2005PolyhedralInvariant,
HirataOhta2008ExactMOAS,
RakovicFiacchini2008MaximalApproximation,
AthanasopoulosBitsoris2009MaximalControlledInvariant,
RunggerTabuada2017RobustControlledInvariant,
EsterhuizenAschenbruckStreif2020MaximalRPI,
WangJungersOng2021MaximalInvariant,
RakovicZhang2023ImplicitMPI,
Wang2023computation,
Ossareh2024complexity}. Although the underlying maximal invariance 
notions differ in whether and how
admissible control inputs and uncertainty are incorporated, they share
the objective of identifying the largest subset of the prescribed
constraint set from which admissibility can be maintained indefinitely.  
Two closely related settings can nevertheless be distinguished.
Controlled invariance allows admissible control actions to be selected
to preserve the constraints, whereas positive invariance requires the
set to be preserved directly by the system dynamics. This article
focuses on the latter setting and, in the presence of model uncertainty 
and process disturbances, on maximal robust positive invariance.

The systematic study of maximal positive invariance for constrained
discrete-time systems can be traced to the classical theory of maximal
output admissible sets for autonomous linear time-invariant systems
\cite{GilbertTan1991MaximalOutputAdmissible}. When output admissibility
is expressed through state constraints, this theory equivalently
characterizes the maximal positively invariant subset of the admissible
state set. Subsequent developments addressed, among other settings,
linear systems with bounded additive disturbances
\cite{KolmanovskyGilbert1998DisturbanceInvariant},
systems with polytopic uncertainty in the state-transition
matrix
\cite{PluymersEtAl2005PolyhedralInvariant},
polynomial systems
\cite{HirataOhta2008ExactMOAS,
RakovicVillanueva2017PolynomialMPI},
piecewise-affine systems
\cite{BenlaoukliEtAl2009PWA},
safe Markov chains
\cite{Janak2022ExactInvariant},
constrained switching and hybrid systems
\cite{AthanasopoulosSmpoukisJungers2017InvariantSets,
AthanasopoulosJungers2018Hybrid},
and time-delay systems
\cite{RakovicGielen2014InvariantFamilies}.

Within this broad literature, additive disturbances and state-transition matrix uncertainty in linear systems have largely been
treated along separate lines. Their simultaneous presence is natural in uncertain and switched linear systems, yet maximal robust positive invariance under both uncertainty mechanisms has received comparatively limited systematic treatment. A closely
related recent letter
\cite{DeyBhasin2024Computation}
considers constrained discrete-time linear time-varying systems
under a prescribed quadratically stabilizing state-feedback
law, with parametric uncertainty and additive disturbances
described by polytopes. Its closed-loop formulation falls within the class of
polytopic linear difference inclusions and is developed under
polyhedral constraints and a common quadratic Lyapunov condition. The present article instead takes linear difference inclusions as the starting point, allowing the invariant-set analysis to be developed without committing at the outset to a particular uncertainty parameterization or constraint representation.

In parallel with this question of modeling generality, the computation of maximal (robust) positively invariant sets raises a separate methodological issue. Several closely related set iterations---referred to in this article as the standard, self-restricted, and incremental iterations and formally defined in Section~\ref{subsec:TMRPIS:iteration}---can be identified across different parts of the literature. Examples of the standard iteration can be found in \cite{RakovicVillanueva2017PolynomialMPI, AthanasopoulosSmpoukisJungers2017InvariantSets, AthanasopoulosJungers2018Hybrid}; the self-restricted iteration is used in \cite{HirataOhta2008ExactMOAS, BorrelliBemporadMorari2017PredictiveControl, WangJungersOng2021MaximalInvariant, Wang2023computation, DeyBhasin2024Computation}; and the incremental iteration appears in \cite{GilbertTan1991MaximalOutputAdmissible, PluymersEtAl2005PolyhedralInvariant}. Both the standard and incremental iterations are presented in \cite{KolmanovskyGilbert1998DisturbanceInvariant, RakovicZhang2023ImplicitMPI}, with the main computational developments in these works following the incremental form. Despite their close relationship, the three iterations have largely been studied separately, and their equivalence and comparative computational properties have not been systematically examined within a common robust framework.

\subsection{Scope and Main Contributions}

Taken together, these modeling and methodological gaps motivate a unified treatment of maximal robust positive invariance directly at the level of linear difference inclusions, accommodating additive disturbances and uncertainty in the state-transition matrix within a common framework.

The first contribution is a general set-theoretic framework that extends classical results on maximal admissible and invariant sets for deterministic linear systems 
\cite{GilbertTan1991MaximalOutputAdmissible}, 
linear systems with additive disturbances 
\cite{KolmanovskyGilbert1998DisturbanceInvariant}, 
and linear systems with polytopic matrix uncertainty 
\cite{PluymersEtAl2005PolyhedralInvariant} to this general linear difference inclusion setting, without imposing a particular uncertainty parameterization or constraint geometry. 
It establishes structural properties of the maximal
robust positively invariant set and its finite iterates
and characterizes the maximal set as the greatest fixed
point of the robust predecessor operator. This leads
to a general descending predecessor iteration and
establishes the equivalence of the standard,
self-restricted, and incremental recursions,
including their exact stopping tests. 

The second contribution establishes conditions for
nonemptiness and finite determination. Under uniform
exponential stability and arbitrary nonempty bounded
disturbances, the Kuratowski upper limit of the
disturbance-reachable sets is shown to be the least
nonempty closed robust positively invariant set.
This gives a necessary and sufficient condition for
nonemptiness under closed state constraints.
Finite determination is further established using
any nonempty closed robust positively invariant set
with a positive containment margin, provided that
at least one iterate is bounded. An explicit upper
bound on the determination index is also derived.

The third contribution is to examine the computational
nonequivalence of the three set-theoretically equivalent
recursions. For polyhedral state constraints and finite
or polytopic matrix families, exact half-space
implementations are derived. These implementations
propagate constraints differently, leading to distinct
complexity growth, storage requirements, and responses
to redundancy removal. The complexity analysis and
numerical study identify the matrix-family size,
determination index, and redundancy-removal strategy
as key factors governing the choice of implementation.

\subsection{Organization}
The remainder of this article is organized as follows.
Section~\ref{sec:ldi} introduces the notation and preliminaries, 
formulates the linear difference inclusion, and discusses representative special cases.
Section~\ref{sec:TMRPIS} develops the theoretical framework for maximal robust positive invariance, including the equivalent set iterations and results on nonemptiness and finite determination.
Section~\ref{sec:algorithmic_computation} presents general set-theoretic algorithms, exact polyhedral implementations, and their computational complexity.
Section~\ref{sec:numerical_study} numerically validates the proposed methods and examines the effects of redundancy removal and problem scaling. Section~\ref{sec:conclusion} concludes the article. 
Proofs of some technically involved results are provided in the Appendix.

\section{Preliminaries and Linear Difference Inclusions}
\label{sec:ldi}

\subsection{Notation and Preliminaries}
The sets of real numbers and nonnegative integers are denoted by 
$\R$ and $\N$, respectively. 
The convex hull, closure, and interior of a set $\mc X$ are denoted by
$\operatorname{conv}(\mc X)$, $\operatorname{cl}(\mc X)$, and 
$\operatorname{int}(\mc X)$, respectively. We denote by $\lceil x\rceil$ 
the smallest integer
greater than or equal to $x\in\R$. For a set $\mc X$, let
$2^{\mc X}:=\{\mc S:\mc S\subseteq\mc X\}$ denote its power set.
A complete lattice is a partially ordered set in which every subset
admits both a supremum and an infimum; see, e.g.,
\cite{Birkhoff1967LatticeTheory,DaveyPriestley2002}.
In particular, $(2^{\mc X},\subseteq)$ is a complete lattice, with
arbitrary unions and intersections giving suprema and infima,
respectively.
Given a self-mapping $\mc T:2^{\mc X}\to2^{\mc X}$, a set
$\mc S\in2^{\mc X}$ is called a fixed point, a post-fixed point, or a
pre-fixed point of $\mc T$ if, respectively,
\[
    \mc T(\mc S)=\mc S,\qquad
    \mc S\subseteq\mc T(\mc S),\qquad
    \mc T(\mc S)\subseteq\mc S.
\]
A polyhedral set is an intersection of finitely many closed half-spaces. 
Hence, it is closed and convex. A bounded polyhedral set is called a polytope.

Given nonempty sets $\mc{X},\mc{Y}\subseteq\R^n$, their Minkowski 
sum and Pontryagin difference are defined as
$\mc{X}+\mc{Y}:=\{x+y:\ x\in\mc{X},\ y\in\mc{Y}\},$
and
$\mc{X}-\mc{Y}:=\{z\in\R^n:\ z+\mc{Y}\subseteq \mc{X}\},$
respectively. For a matrix, or scalar, $M$ of compatible dimensions, 
the image and preimage of a set $\mc{X}$ under $M$ are denoted by 
$M\mc{X}:=\{Mx:\ x\in\mc{X}\}$ and  $ M^{-1}\mc{X}:=\{x:\ Mx\in\mc{X}\},$ 
respectively. For a point $x\in\R^n$ and a nonempty set $\mc S\subseteq\R^n$, 
the distance from $x$ to $\mc S$ is defined by
\[
    \operatorname{dist}(x,\mc S)
    :=
    \inf_{s\in\mc S}\|x-s\|.
\] The support function $\supportbb{\mc X}:\R^n\to[-\infty,+\infty]$ of a set
$\mc X\subseteq\R^n$ is defined, for every direction $y\in\R^n$, by
\begin{equation*}
\support{\mc X}{y}:=\sup_{x\in\mc X} y^\top x .
\end{equation*}
We extend this notation row-wise to matrices. Namely, for a matrix
$Y\in\R^{q\times n}$, let $y_i^\top$ denote the $i$-th row of $Y$ for all 
$i=1,\ldots,q$, and let 
\begin{equation*}
    \support{\mc X}{Y}
    :=
    \begin{bmatrix}
    \support{\mc X}{y_1}\\
    \vdots\\
    \support{\mc X}{y_q}
    \end{bmatrix}.
\end{equation*}

\subsection{Linear Difference Inclusion}
We consider uncertain linear discrete-time systems subject to additive
disturbances and uncertainty in the state-transition matrix, described
by the linear difference inclusion
\begin{align} 
    &x^+\in\mc{F}(x) \text{ with}\nonumber\\ 
    \label{eq:pre:dynamics} 
    &\mc{F}(x) := \{Ax+w: A\in\mathscr A,\ w\in\mc{W}\}, 
\end{align}
where $x\in\R^n$ and $x^+\in\R^n$ denote the current and
successor states, respectively. The admissible state-transition
matrices are collected in a nonempty known family
$\mathscr A\subseteq\R^{n\times n}$, and the additive disturbance
takes values in a nonempty known set $\mc W\subseteq\R^n$. 
Thus, the uncertainty represented by $\mathscr A$ acts multiplicatively
on the state, whereas $\mc W$ describes the additive disturbance.
At each time instant, any pair
$(A,w)\in\mathscr A\times\mathcal W$ is admissible,
and $A$ and $w$ may vary independently and
arbitrarily over time.
The state is required to satisfy the hard constraint
\begin{equation}
\label{eq:pre:constraints}
    x\in\mc X,
\end{equation}
where $\mc X\subseteq\R^n$ is a prescribed constraint set.

The linear difference inclusion~\eqref{eq:pre:dynamics} encompasses several
standard classes of uncertain and switched discrete-time linear
systems. For example, consider a linear parameter-varying or 
parameter-uncertain system with unrestricted parameter variation,
\begin{equation*}
    x^+=A(\theta)x+w,
    \qquad
    \theta\in\Theta,\quad w\in\mc W,
\end{equation*}
where $\theta$ is a scheduling or uncertain parameter. This system is
represented by~\eqref{eq:pre:dynamics} with
$\mathscr A:=\{A(\theta):\theta\in\Theta\}.$
The parameter $\theta$, and hence the matrix $A(\theta)$, may vary
arbitrarily from one time instant to the next.

A prominent special case is the polytopic uncertain system
\begin{align*}
    &x^+ = Ax+w,\\
    &A= \sum_{i=1}^{s}\lambda_i\overline A_i,
    \qquad
    \lambda_i\geq0,\quad
    \sum_{i=1}^{s}\lambda_i=1,\quad
    w\in\mc W,
\end{align*}
where
$\overline A_1,\ldots,\overline A_s\in\R^{n\times n}$
are the vertex matrices. This system is represented by choosing
$\mathscr A := \operatorname{conv} \{\overline A_1,\ldots,\overline A_s\}.$
Thus, every admissible matrix is a convex combination of the vertex
matrices. The coefficients $\lambda_i$, and hence the matrix $A$, may
vary arbitrarily over time subject to the simplex constraints above.

When $\mathscr A=\{\overline A_1,\ldots,\overline A_s\}$,
then~\eqref{eq:pre:dynamics} becomes the switched linear system with
additive disturbances
\begin{equation*}
    x^+=\overline A_{\sigma}x+w,
    \qquad
    \sigma\in\{1,\ldots,s\},\quad
    w\in\mc W,
\end{equation*}
with arbitrary switching. 
If $\mathscr A=\{A\}$, then~\eqref{eq:pre:dynamics} reduces to
the linear time-invariant system with additive disturbance
\begin{equation*}
    x^+=Ax+w,
    \qquad
    w\in\mc W.
\end{equation*}
If, in addition, $\mc W=\{0\}$, it reduces to the deterministic 
linear system
\begin{equation*}
    x^+=Ax.
\end{equation*}

\section{The Maximal Robust Positively Invariant Set}
\label{sec:TMRPIS}
This section develops the theoretical foundations for characterizing
the maximal robust positively invariant set of~\eqref{eq:pre:dynamics}
subject to the state constraint~\eqref{eq:pre:constraints}.

\subsection{Definition and Characterization}
For the linear difference inclusion 
considered in this article, robust positive invariance is defined as
follows.
\begin{definition}
\label{def:rpi}
    A set $\mc{S} \subseteq \R^n$ is said to be \emph{robust positively
    invariant} for the dynamics $x^+\in\mc{F}(x)$ under the state constraint 
    $x \in \mc{X}$ if and only if, for all $x \in \mc{S}$, it holds that 
    $\mc{F}(x) \subseteq \mc{S}$ and $x \in \mc{X}$.
\end{definition}

For every set $\mc S\subseteq\R^n$, define its one-step image under the
difference inclusion \eqref{eq:pre:dynamics} by
\begin{equation}
\label{eq:FS}
    \mc F(\mc S)
    :=
    \{Ax+w:
    x\in\mc S,\ A\in\mathscr A,\ w\in\mc W\}.
\end{equation}
Robust positive invariance of a set $\mc S$, as defined in
Definition~\ref{def:rpi}, is therefore equivalent to the chained set
inclusions
\begin{equation}
\label{eq:rpi_chained_inclusions}
    \mc F(\mc S)\subseteq\mc S\subseteq\mc X.
\end{equation}
For every target set $\mc S\subseteq\R^n$, define its one-step
robust predecessor within $\mc X$ by
\begin{align}
    \mc F^{-1}(\mc S)
    &:=
    \{x\in\R^n:\mc F(x)\subseteq\mc S\}\bigcap\mc X
    \nonumber\\
    &=
    \left(
    \bigcap_{A\in\mathscr A}
    A^{-1}(\mc S-\mc W)
    \right)
    \bigcap\mc X .
    \label{eq:predecessorb}
\end{align}
Accordingly, the robust predecessor operator is a set-to-set mapping 
$\mc F^{-1}:2^{\R^n}\to2^{\mc X}.$
Since $\mc F^{-1}(\mc S)\subseteq\mc X$ for every
$\mc S\subseteq\R^n$, restricting its domain to $2^{\mc X}$ yields the
self-mapping 
$\mc F^{-1}:2^{\mc X}\to2^{\mc X}.$
For a given target set $\mc S$, the set $\mc F^{-1}(\mc S)$ consists of
all states $x$ in $\mc X$ from which every admissible one-step successor 
$x^+$ lies
in $\mc S$. In view of~\eqref{eq:predecessorb}, robust positive invariance 
of a set $\mc S$, as expressed in~\eqref{eq:rpi_chained_inclusions}, 
is equivalent to the predecessor inclusion
\begin{equation}
\label{eq:pre:S_subseteqF-1S}
    \mc S\subseteq\mc F^{-1}(\mc S).
\end{equation}
Thus, robust positive
invariance of $\mc S$ is equivalent to $\mc S$ being a post-fixed point
of the robust predecessor operator $\mc F^{-1}$.

The notion of the maximal robust positively invariant set utilized throughout 
this article is given by the following.
\begin{definition}
    A set $\mc M\subseteq\R^n$ is called the \emph{maximal robust
    positively invariant set} for the dynamics $x^+\in\mc{F}(x)$ under 
    the state constraint $x \in \mc{X}$ if and only if it is robust positively 
    invariant and contains every other robust positively invariant set for 
    the same dynamics and state constraint.
\end{definition}

In what follows, for typographical convenience, we use the terms 
(maximal) robust positive invariance and (maximal) robust positively 
invariant set without repeatedly specifying their dependence on the 
dynamics $x^+\in\mc F(x)$ and the state constraint $x\in\mc X$, 
since this dependence is fixed and clear from the context.

The robust predecessor operator $\mc F^{-1}$ has a basic order-preserving
property. 
For any $\mc S_1,\mc S_2\subseteq\mc X$ with
$\mc S_1\subseteq\mc S_2$, if
$x\in\mc F^{-1}(\mc S_1)$, then
$x\in\mc X$ and
$\mc F(x)\subseteq\mc S_1\subseteq\mc S_2$.
Hence,
it follows that 
\begin{equation}
    \mc F^{-1}(\mc S_1)
    \subseteq
    \mc F^{-1}(\mc S_2),
\end{equation}
so that
$\mc F^{-1}:2^{\mc X}\to2^{\mc X}$ is monotone.
Together with the complete-lattice structure of
$(2^{\mc X},\subseteq)$, this monotonicity leads to a natural
fixed-point characterization of the maximal robust positively
invariant set. The following theorem establishes its existence
and identifies it as the greatest fixed point of
$\mc F^{-1}$ with respect to set inclusion. 

\begin{theorem}
\label{thm:fixedpoint}
    The robust predecessor operator $\mc F^{-1}$ admits a greatest fixed
    point $\mc M$ with respect to set inclusion. In particular,
    \begin{equation}
    \label{eq:M_fixedpoint}
        \mc M=\mc F^{-1}(\mc M),
    \end{equation}
    and
    \begin{equation}
    \label{eq:gfp_postfixed}
        \mc M
        =
        \bigcup
        \left\{
        \mc S\subseteq\mc X:
        \mc S\subseteq\mc F^{-1}(\mc S)
        \right\}.
    \end{equation}
    Furthermore, $\mc M$ coincides with the maximal robust positively
    invariant set.
\end{theorem}
\begin{proof}
    Since $(2^{\mc X},\subseteq)$ is a complete lattice and
    $\mc F^{-1}:2^{\mc X}\to2^{\mc X}$ is monotone,
    the Knaster--Tarski fixed-point theorem
    \cite{Tarski1955LatticeFixpoint,DaveyPriestley2002}
    guarantees that $\mc F^{-1}$ admits a greatest fixed point $\mc M$ and
    characterizes it as
    $\mc M
    =
    \sup
    \left\{
    \mc S\in2^{\mc X}:
    \mc S\subseteq\mc F^{-1}(\mc S)
    \right\}.$
    Since suprema in $(2^{\mc X},\subseteq)$ are given by set union, we have
    $\mc M
    =
    \bigcup
    \left\{
    \mc S\subseteq\mc X:
    \mc S\subseteq\mc F^{-1}(\mc S)
    \right\},$
    which proves~\eqref{eq:gfp_postfixed}.
    Since $\mc M$ is a fixed point, we have
    $\mc M=\mc F^{-1}(\mc M)$, which proves~\eqref{eq:M_fixedpoint}.

    It remains to identify $\mc M$ with the maximal robust positively
    invariant set. By~\eqref{eq:pre:S_subseteqF-1S}, a set
    $\mc S\subseteq\mc X$ is robust positively invariant if and only if
    it is a post-fixed point of $\mc F^{-1}$. Since $\mc M$ is a fixed
    point, it is, in particular, a post-fixed point and is therefore robust
    positively invariant. Conversely, every robust positively invariant set
    is a post-fixed point of $\mc F^{-1}$ and is consequently contained in
    $\mc M$ by~\eqref{eq:gfp_postfixed}. Thus, $\mc M$ is the maximal
    robust positively invariant set.
\end{proof}
The next results establish some regularity properties of the
maximal robust positively invariant set.

\begin{corollary}
\label{cor:closed_M}
    If $\mc X\subseteq\R^n$ is closed, then the maximal robust positively
    invariant set $\mc M$ is closed. In particular, if $\mc X$ is
    compact, then $\mc M$ is compact.
\end{corollary}
\begin{proof}
    Since $\mc M\subseteq\mc X$ and $\mc X$ is closed, one has $\operatorname{cl}(\mc M)\subseteq\mc X$.
    We show that $\operatorname{cl}(\mc M)$ is robust positively invariant.
    Let $x\in\operatorname{cl}(\mc M)$. Then there exists a sequence
    $\{x_j\}_{j\in\N}\subseteq\mc M$ such that $x_j\to x$. For arbitrary
    $A\in\mathscr A$ and $w\in\mc W$, robust positive invariance of
    $\mc M$ gives
    $Ax_j+w\in\mc M,
        \ j\in\N.$
    Since the mapping $x\mapsto Ax+w$ is continuous, one has
    $Ax_j+w\to Ax+w.$
    Therefore, we have $Ax+w\in\operatorname{cl}(\mc M).$
    Because $A\in\mathscr A$ and $w\in\mc W$ were arbitrary, it follows that $\mc F(x)\subseteq\operatorname{cl}(\mc M).$
    Together with
    $x\in\operatorname{cl}(\mc M)\subseteq\mc X$, this proves that
    $\operatorname{cl}(\mc M)$ is robust positively invariant.
    By maximality of $\mc M$, we have $\operatorname{cl}(\mc M)\subseteq\mc M.$
    The reverse inclusion always holds, and hence $\operatorname{cl}(\mc M)=\mc M.$
    Therefore, $\mc M$ is closed. If $\mc X$ is compact, then $\mc M$ is a closed subset 
    of the compact set $\mc X$ and is therefore compact.
\end{proof}
\begin{corollary}
\label{cor:convex_M}
    If $\mc X\subseteq\R^n$ is convex, then the maximal robust positively
    invariant set $\mc M$ is convex.
\end{corollary}
\begin{proof}
    Since $\mc M\subseteq\mc X$ and $\mc X$ is convex, we have $\operatorname{conv}(\mc M)\subseteq\mc X$.
    We show that $\operatorname{conv}(\mc M)$ is robust positively
    invariant. Let
    $x=\sum_{i=1}^{q}\lambda_i x_i
        \in\operatorname{conv}(\mc M),$
    where $x_i\in\mc M$, $\lambda_i\geq0$, and
    $\sum_{i=1}^{q}\lambda_i=1$. For arbitrary
    $A\in\mathscr A$ and $w\in\mc W$, one has
    $Ax+w = \sum_{i=1}^{q}\lambda_i(Ax_i+w).$
    Since $\mc M$ is robust positively invariant, we have
    $Ax_i+w\in\mc M,
        \ i=1,\ldots,q.$
    It follows that
    $Ax+w\in\operatorname{conv}(\mc M).$
    Thus, $\operatorname{conv}(\mc M)$ is robust positively invariant.
    By the maximality of $\mc M$, one has $\operatorname{conv}(\mc M)\subseteq\mc M.$
    Since the reverse inclusion holds by definition,
    one has $\operatorname{conv}(\mc M)=\mc M.$
    Thus, $\mc M$ is convex.
\end{proof}

\begin{corollary}
    If $\mc X$ and $\mc W$ are centrally symmetric about the origin,
    i.e., $\mc X=-\mc X$ and $\mc W=-\mc W$, then the maximal robust
    positively invariant set $\mc M$ is centrally symmetric about the
    origin, i.e., $\mc M=-\mc M$.
\end{corollary}

\begin{proof}
    Since $\mc M\subseteq\mc X$ and $\mc X=-\mc X$, one has
    $-\mc M\subseteq\mc X$. We show that $-\mc M$ is robust positively
    invariant. Let $x\in-\mc M$. Then $-x\in\mc M$. For arbitrary
    $A\in\mathscr A$ and $w\in\mc W$, symmetry of $\mc W$ gives
    $-w\in\mc W$. Since $\mc M$ is robust positively invariant, we have
    $A(-x)+(-w)\in\mc M.$
    Taking negatives yields $-(A(-x)+(-w))=Ax+w\in-\mc M.$
    Hence, $-\mc M$ is robust positively invariant. By maximality of
    $\mc M$, one has $-\mc M\subseteq\mc M$. Taking negatives gives the
    reverse inclusion, and therefore $\mc M=-\mc M$.
\end{proof}

\begin{corollary}
Let $\mc M(\mathscr A,\mc W,\mc X)$ denote the maximal robust
positively invariant set associated with
$(\mathscr A,\mc W,\mc X)$. Then, for any $\alpha>0$, we have 
\begin{equation}
    \mc M(\mathscr A,\alpha\mc W,\alpha\mc X)
    =
    \alpha\mc M(\mathscr A,\mc W,\mc X).
\end{equation}
Equivalently, we have
\begin{equation}
    \mc M(\mathscr A,\mc W,\alpha\mc X)
    =
    \alpha\mc M(
        \mathscr A,\frac{1}{\alpha}\mc W,\mc X).
\end{equation}
\end{corollary}

\begin{proof}
Let $\mc M:=\mc M(\mathscr A,\mc W,\mc X)$. For any
$x=\alpha y\in\alpha\mc M$, $A\in\mathscr A$, and
$w=\alpha v\in\alpha\mc W$, we have $y\in\mc M$ and $v\in\mc W$, so that
$Ax+w
    =
    \alpha(Ay+v)
    \in
    \alpha\mc M.$
Thus, $\alpha\mc M$ is robust positively invariant in
$\alpha\mc X$, and hence it holds that 
$\alpha\mc M
    \subseteq
    \mc M(\mathscr A,\alpha\mc W,\alpha\mc X).$
Applying the same argument with scaling factor $\frac{1}{\alpha}$ gives
the reverse inclusion. The second relation follows by replacing
$\mc W$ with $\frac{1}{\alpha}\mc W$.
\end{proof}

\subsection{Descending Set Iteration and Equivalent Recursions}
\label{subsec:TMRPIS:iteration}
Theorem~\ref{thm:fixedpoint} identifies the maximal robust positively
invariant set $\mc M$ as the greatest fixed point of the robust
predecessor operator $\mc F^{-1}$. This fixed-point characterization
naturally leads to a descending Kleene-type set iteration for computing
$\mc M$ (see, e.g., \cite{Cousot1979,StoltenbergHansen1994}). In contrast
to the usual ascending Kleene iteration for least fixed points, the
present iteration starts from a pre-fixed point
of $\mc F^{-1}$ containing $\mc M$ and
generates a decreasing sequence of sets. In addition to being monotone, 
the robust predecessor operator $\mc F^{-1}$
preserves arbitrary nonempty intersections.  Indeed, for every nonempty
index set $\mathcal I$, every collection
$\{\mc S_i\}_{i\in\mathcal I}\subseteq 2^{\mc X}$, and every
$x\in\mathbb R^n$,
\[
    \begin{aligned}
        x\in\mc F^{-1}\left(\bigcap_{i\in\mathcal I}\mc S_i\right)
        &\iff
        x\in\mc X
        \ \text{and}\
        \mc F(x)\subseteq\mc S_i,
        \quad\forall i\in\mathcal I
        \\
        &\iff
        x\in\bigcap_{i\in\mathcal I}\mc F^{-1}(\mc S_i).
    \end{aligned}
\]
Hence, we have  
\begin{equation}
\label{eq:predecessor_intersection}
    \mc F^{-1}\left(\bigcap_{i\in\mathcal I}\mc S_i\right)
    =
    \bigcap_{i\in\mathcal I}\mc F^{-1}(\mc S_i).
\end{equation}
Together with the recursive construction and the decreasing property
of the iterates, this property ensures that their intersection is a
fixed point. The following theorem shows that this fixed point is
$\mc M$. 

\begin{theorem}
\label{thm:pre-fixed-iteration}
    Let $\mc M$ denote the greatest fixed point of $\mc F^{-1}$, and let
    $\mc P\in 2^{\mc X}$ satisfy
    \begin{equation}
    \label{eq:prefixed_upper_bound}
        \mc M\subseteq\mc P
        \qquad \text{and} \qquad 
        \mc F^{-1}(\mc P)\subseteq\mc P.
    \end{equation}
    Consider the sequence $\{\mc P_k\}_{k\in\N}$ defined by
    \begin{equation}
    \label{eq:fixed_point_sequence}
        \mc P_0:=\mc P,
        \qquad
        \mc P_{k+1}:=\mc F^{-1}(\mc P_k),
        \quad k\in\N.
    \end{equation}
    Then:
    \begin{enumerate}
    \renewcommand{\labelenumi}{(\roman{enumi})}
        \item the sequence is decreasing and contains $\mc M$, namely,
        \begin{equation}
            \mc M\subseteq\mc P_{k+1}\subseteq\mc P_k,
            \qquad k\in\N;
        \end{equation}

        \item the intersection of the sequence equals $\mc M$, namely,
        \begin{equation}
        \label{eq:prefixed_iteration_limit}
            \mc M
            =
            \mc P_\infty
            :=
            \bigcap_{k\in\N}\mc P_k.
        \end{equation}
    \end{enumerate}
\end{theorem}
\begin{proof}
    Since $\mc F^{-1}$ is monotone and
    $\mc M=\mc F^{-1}(\mc M)$, the inclusion
    $\mc M\subseteq\mc P_0$ implies
    $\mc M = \mc F^{-1}(\mc M) \subseteq \mc F^{-1}(\mc P_0) = \mc P_1.$
    Repeating this argument inductively gives
    $\mc M\subseteq\mc P_k,\ k\in\N.$
    By~\eqref{eq:prefixed_upper_bound},
    $\mc P_1\subseteq\mc P_0$. If
    $\mc P_k\subseteq\mc P_{k-1}$ for some $k\geq1$, then monotonicity of
    $\mc F^{-1}$ yields
    $\mc P_{k+1}
        =
        \mc F^{-1}(\mc P_k)
        \subseteq
        \mc F^{-1}(\mc P_{k-1})
        =
        \mc P_k.$
    Thus, $ \mc M\subseteq\mc P_{k+1}\subseteq\mc P_k,\ k\in\N.$
    Let $\mc P_\infty:=\bigcap_{k\in\N}\mc P_k.$ By the intersection-preserving property
    \eqref{eq:predecessor_intersection}, we have
    \begin{align*}
        \mc F^{-1}(\mc P_\infty)
        =
        \bigcap_{k\in\N}\mc F^{-1}(\mc P_k)=
        \bigcap_{k\in\N}\mc P_{k+1}=
        \mc P_\infty,
    \end{align*}
    where the last equality follows because the sequence is decreasing.
    Hence, $\mc P_\infty$ is a fixed point of $\mc F^{-1}$.
    Since $\mc M$ is the greatest fixed point, we have 
    $\mc P_\infty\subseteq\mc M.$
    On the other hand, $\mc M\subseteq\mc P_k$ for every $k\in\N$, and
    therefore
    $\mc M\subseteq\mc P_\infty.$
    It follows that
    $\mc P_\infty=\mc M.$
\end{proof}

The state constraint set $\mc X$ provides the canonical initialization
of~\eqref{eq:fixed_point_sequence}. Indeed, every robust positively
invariant set is contained in $\mc X$, and hence
$\mc M\subseteq\mc X$. Moreover, the definition of the robust
predecessor operator gives
$\mc F^{-1}(\mc X)\subseteq\mc X.$
Thus, $\mc X$ is a pre-fixed point
of $\mc F^{-1}$ containing $\mc M$. 
Accordingly, we define the sequence  
$\{\mc X_k\}_{k\in\N}$ by
\begin{equation}
\label{eq:comp:iteration}
    \mc X_0:=\mc X,\qquad 
    \mc X_{k+1}:=\mc F^{-1}(\mc X_k),
    \quad k\in\N.
\end{equation}
By~\eqref{eq:predecessorb}, the recursion can be written explicitly as
\begin{equation*}
    \mc X_{k+1}
    =
    \left(
    \bigcap_{A\in\mathscr A}
    A^{-1}(\mc X_k-\mc W)
    \right)
    \bigcap\mc X,
    \qquad k\in\N.
\end{equation*}
This update retains precisely those states in $\mc X$ from which every
admissible one-step successor belongs to $\mc X_k$. Repeated application
of this one-step predecessor construction yields the following $k$-step
trajectory characterization of $\mc X_k$:
    \begin{align}
        \mc X_k= \bigl\{x_0\in\R^n \,:\ x_j\in\mc X,\  
        &j=0,1,\ldots,k,\nonumber\\
        &\forall A_0,\ldots,A_{k-1}\in\mathscr A,\nonumber\\ 
        \label{eq:comp:xk_finite_horizon}
        &\forall w_0,\ldots,w_{k-1}\in\mc W \bigr\},    
    \end{align}
where the sequence $\{x_j\}_{j=0}^{k}$ is generated by 
$x_{j+1}=A_jx_j+w_j,\ j=0,1,\ldots,k-1$. 
Thus, $\mc X_k$ consists of all initial states from which every
admissible $k$-step trajectory remains in $\mc X$. 

Next, set $\mc Y_0:=\mc X$. For each $k\geq1$, define $\mc Y_k$ as the set
of initial states for which every admissible trajectory satisfies the
state constraint at time $k$:
\begin{align}
    \mc Y_k
    :=
    \bigl\{
    x_0\in\R^n:\ x_k\in\mc X,\ 
    &\forall A_0,\ldots,A_{k-1}\in\mathscr A,\nonumber\\
    \label{eq:Y_k}
    &\forall w_0,\ldots,w_{k-1}\in\mc W
    \bigr\},
\end{align}
where $x_k$ is generated by the recursion
$x_{j+1}=A_jx_j+w_j$, $j=0,1,\ldots,k-1$.

For each $k\geq1$, let $\mathfrak B_k(\mathscr A)$ denote the set of
matrix-product tuples
\begin{equation}
\label{eq:Psi}
    \boldsymbol{\Psi}_k
    =
    (\Psi_{k,0},\ldots,\Psi_{k,k})
\end{equation}
generated by admissible matrix sequences
$A_0,\ldots,A_{k-1}\in\mathscr A$ according to
\begin{equation*}
    \Psi_{k,k}:=I,
    \qquad
    \Psi_{k,i}:=A_{k-1}A_{k-2}\cdots A_i,
    \quad i=0,\ldots,k-1.
\end{equation*}
Starting from $x_0$, the state reached after $k$ steps along the
corresponding admissible matrix and disturbance sequences is
\begin{equation*}
    x_k
    =
    \Psi_{k,0}x_0
    +
    \sum_{i=0}^{k-1}\Psi_{k,i+1}w_i.
\end{equation*}
Consequently, for each $k\geq 1$, the set $\mc Y_k$ defined
in~\eqref{eq:Y_k} admits the equivalent representation
\begin{equation}
\label{eq:comp:Y_k}
    \mc Y_k
    =
    \bigcap_{\boldsymbol{\Psi}_k\in\mathfrak B_k(\mathscr A)}
    \Psi_{k,0}^{-1}
    \left(
    \mc X-\sum_{i=0}^{k-1}\Psi_{k,i+1}\mc W
    \right).
\end{equation}
The set $\mc Y_k$ imposes the
state constraint only at time $k$, whereas $\mc X_k$ imposes the state
constraint at every time from $0$ to $k$. Hence, by
\eqref{eq:comp:xk_finite_horizon} and \eqref{eq:Y_k}, we have
\begin{equation}
\label{eq:finite:XkYk}
    \mc X_k = \bigcap_{j=0}^{k}\mc Y_j,\qquad k\in\N.
\end{equation}
The identity~\eqref{eq:finite:XkYk} gives an incremental representation
of the set iteration: passing from $\mc X_k$ to $\mc X_{k+1}$ amounts
to imposing the additional constraint encoded by $\mc Y_{k+1}$.
For each $k\in\N$, define the intermediate predecessor set
\begin{equation}
\label{eq:intermediate_predecessor}
    \overline{\mc X}_{k+1}
    :=
    \bigcap_{A\in\mathscr A}
    A^{-1}(\mc X_k-\mc W).
\end{equation}
The set $\overline{\mc X}_{k+1}$ contains all states from which every
admissible one-step successor belongs to $\mc X_k$, without imposing
the additional intersection with the state constraint $\mc X$.
Together with the decreasing property of
$\{\mc X_k\}_{k\in\N}$, these observations lead to the following
equivalent forms of the set iteration.

\begin{proposition}
\label{prop:equivalent_iterations}
    Starting from $\mc X_0=\mc X$, the following three set-iteration
    schemes generate the same sequence $\{\mc X_k\}_{k\in\N}$:
    \begin{enumerate}
    \renewcommand{\labelenumi}{(\roman{enumi})}
        \item \emph{Standard set iteration:}
        \begin{equation}
        \label{eq:equiv_iteration_1}
            \mc X_{k+1}
            =
            \overline{\mc X}_{k+1}
            \bigcap\mc X,\quad k\in\N.
        \end{equation}

        \item \emph{Self-restricted set iteration:}
        \begin{equation}
        \label{eq:equiv_iteration_2}
            \mc X_{k+1}
            =
            \overline{\mc X}_{k+1}
            \bigcap\mc X_k,\quad k\in\N.
        \end{equation}

        \item \emph{Incremental set iteration:}
        \begin{equation}
        \label{eq:equiv_iteration_3}
            \mc X_{k+1}
            =
            \mc X_k\bigcap\mc Y_{k+1},\quad k\in\N.
        \end{equation}
    \end{enumerate}
\end{proposition}

\begin{proof}
    By Theorem~\ref{thm:pre-fixed-iteration} with $\mc P=\mc X$, the
    standard iteration generates a decreasing sequence. Hence, it holds that
    $\mc X_{k+1}\subseteq\mc X_k\subseteq\mc X.$
    Using
    $\mc X_{k+1}=\overline{\mc X}_{k+1}\bigcap\mc X$, one obtains
    \[
        \overline{\mc X}_{k+1}\bigcap\mc X
        =
        \left(
            \overline{\mc X}_{k+1}\bigcap\mc X
        \right)\bigcap\mc X_k
        =
        \overline{\mc X}_{k+1}\bigcap\mc X_k.
    \]
    Thus, the standard and self-restricted iterations generate the same
    update.
    Moreover, by \eqref{eq:finite:XkYk},
    \[
        \mc X_{k+1}
        =
        \bigcap_{j=0}^{k+1}\mc Y_j
        =
        \big(\bigcap_{j=0}^{k}\mc Y_j\big)\bigcap\mc Y_{k+1}
        =
        \mc X_k\bigcap\mc Y_{k+1}.
    \]
    Thus, the incremental iteration generates the same sequence as the standard
    iteration. Consequently, the three recursions
    \eqref{eq:equiv_iteration_1}--\eqref{eq:equiv_iteration_3} are equivalent.
\end{proof}

The following corollary summarizes the principal properties of the
common sequence generated by the three equivalent recursions.
\begin{corollary}
\label{cor:MRPI}
    Consider the sequence $\{\mc X_k\}_{k\in\N}$ generated by any one
    of the equivalent iterations
    \eqref{eq:equiv_iteration_1}--\eqref{eq:equiv_iteration_3}, with
    $\mc X_0=\mc X$. Then:
    \begin{enumerate}
    \renewcommand{\labelenumi}{(\roman{enumi})}
        \item the sequence is decreasing and contains $\mc M$, namely,
        \begin{equation}
        \label{eq:Xk_monotonicity}
            \mc M
            \subseteq
            \mc X_{k+1}
            \subseteq
            \mc X_k,
            \qquad k\in\N;
        \end{equation}

        \item its intersection equals the maximal robust positively
        invariant set:
        \begin{equation}
        \label{eq:MRPI_SetIteration}
            \mc M
            =
            \mc X_\infty
            :=
            \bigcap_{k\in\N}\mc X_k;
        \end{equation}

        \item if $\mc X$ is closed, $\mc X_k$ and $\mc Y_k$ are
        closed for every $k\in\N$;

        \item if $\mc X$ is convex, $\mc X_k$ and $\mc Y_k$ are
        convex for every $k\in\N$;

        \item if $\mc X=-\mc X$ and $\mc W=-\mc W$, then $\mc X_k=-\mc X_k$ 
        and $\mc Y_k=-\mc Y_k$ for every $k\in\N$.
    \end{enumerate}
\end{corollary}

\begin{proof}
    Statements~(i) and~(ii) follow directly from
    Theorem~\ref{thm:pre-fixed-iteration} with $\mc P=\mc X$ and
    Proposition~\ref{prop:equivalent_iterations}.

    Suppose that $\mc X$ is closed. For every fixed admissible matrix and 
    disturbance sequence, the state
    reached after $k$ steps is a continuous affine function of the initial
    state $x_0$. By~\eqref{eq:Y_k}, the set
    $\mc Y_k$ is an intersection of preimages of $\mc X$ under such
    mappings and is therefore closed. Since $\mc Y_0=\mc X$ and
    $\mc X_k=\bigcap_{j=0}^{k}\mc Y_j,$
    the set $\mc X_k$ is also closed.

    If $\mc X$ is convex, the same argument applies because affine
    preimages and arbitrary intersections preserve convexity. Hence,
    $\mc X_k$ and $\mc Y_k$ are convex for every $k\in\N$.

    If $\mc X=-\mc X$ and $\mc W=-\mc W$, linearity of the dynamics
    implies that the negative of every trajectory generated from $x_0$
    under an admissible disturbance sequence is a trajectory generated
    from $-x_0$ under the corresponding negative disturbance sequence.
    Hence, $\mc Y_k=-\mc Y_k$ for every $k\in\N$. Since
    $\mc X_k=\bigcap_{j=0}^k\mc Y_j$, it follows that
    $\mc X_k=-\mc X_k$.
\end{proof}

\subsection{Exact Stopping Tests for Finite Determination}
The decreasing sequence $\{\mc X_k\}_{k\in\mathbb N}$ need not become
stationary after finitely many iterations. When it does, $\mc M$ is
obtained exactly from a finite iterate. 
\begin{definition}
    The maximal robust positively invariant set $\mc M$ is said to be
    finitely determined if there exists an index $k\in\N$ such that
    $\mc M=\mc X_k.$
    The smallest index for which $\mc M=\mc X_k$ holds is called the 
    determination index.
\end{definition}

The next result gives necessary and sufficient conditions under which
$\mc M$ is finitely determined by any one of the equivalent iterations
\eqref{eq:equiv_iteration_1}--\eqref{eq:equiv_iteration_3}.

\begin{theorem}
\label{thm:finite_determination}
    Consider the sequence $\{\mc X_k\}_{k\in\N}$ generated by any one
    of the equivalent iterations
    \eqref{eq:equiv_iteration_1}--\eqref{eq:equiv_iteration_3}, with
    $\mc X_0=\mc X$. 
    Consider the sequence $\{\overline{\mc X}_{k}\}_{k\ge1}$ defined 
    in~\eqref{eq:intermediate_predecessor}.
    Consider also the sequence
    $\{\mc Y_k\}_{k\in\N}$ defined by $\mc Y_0=\mc X$
    and~\eqref{eq:comp:Y_k}. 
    Then the following statements are equivalent:
    \begin{enumerate}
    \renewcommand{\labelenumi}{(\roman{enumi})}
        \item The maximal robust positively invariant set $\mc M$ is
        finitely determined.

        \item There exists $k\in\N$ such that
        \begin{equation}
        \label{eq:finite:stationarity_condition}
            \mc X_k\subseteq\mc X_{k+1}.
        \end{equation}

        \item There exists $k\in\N$ such that
        \begin{equation}
        \label{eq:finite:stationarity_condition_}
            \mc X_k\subseteq\overline{\mc X}_{k+1}.
        \end{equation}

        \item There exists $k\in\N$ such that
        \begin{equation}
        \label{eq:finite:incremental_condition}
            \mc X_k\subseteq\mc Y_{k+1}.
        \end{equation}
    \end{enumerate}
    Moreover, if any one of
    \eqref{eq:finite:stationarity_condition},
    \eqref{eq:finite:stationarity_condition_}, or
    \eqref{eq:finite:incremental_condition}
    holds for some $k\in\N$, then
    $\mc M=\mc X_\infty=\mc X_k=\mc X_{k+1}.$
\end{theorem}
\begin{proof}
    The proof can be found in Appendix~\ref{app:finite_determination}.
\end{proof}

The inclusions
\eqref{eq:finite:stationarity_condition}--\eqref{eq:finite:incremental_condition}
provide three equivalent exact stopping tests for finite determination.
Condition~\eqref{eq:finite:stationarity_condition} compares two
consecutive iterates. Since the sequence is decreasing, this condition
is equivalent to
$ \mc X_k=\mc X_{k+1}.$
Condition~\eqref{eq:finite:stationarity_condition_} checks whether
$\mc X_k$ is robust positively invariant, since
$\mc X_k\subseteq\overline{\mc X}_{k+1}$ is equivalent to
$\mc F(\mc X_k)\subseteq\mc X_k$. This test can be performed before
the auxiliary intersection used to construct $\mc X_{k+1}$. Finally,
condition~\eqref{eq:finite:incremental_condition} tests whether the
additional constraint represented by $\mc Y_{k+1}$ is redundant on the
current iterate $\mc X_k$.

\subsection{Nonemptiness and Finite Determination}
The existence of the maximal robust positively invariant set does not
preclude $\mc M=\emptyset$. This subsection establishes conditions
for its nonemptiness and finite determination.

We first recall the stability notion used below; see, e.g.,
\cite{Shorten2007Stability,Wirth2002JSR}. The matrix family
$\mathscr A$ is said to be uniformly exponentially stable if there
exist constants $c\geq1$ and $\rho\in[0,1)$ such that
\begin{equation}
\label{eq:uniform_exponential_stability}
    \|A_{k-1}\cdots A_1A_0\|
    \leq
    c\rho^k
\end{equation}
for every $k\geq1$ and every
$A_0,\ldots,A_{k-1}\in\mathscr A$, where $\|\cdot\|$ denotes a fixed
matrix norm induced by a vector norm. Thus, every admissible matrix 
product generated by $\mathscr A$
decays exponentially, uniformly over all admissible matrix sequences.
Equivalently, in terms of the matrix-product tuples introduced in 
\eqref{eq:Psi}, for every $k\geq1$, every
$\boldsymbol{\Psi}_k=(\Psi_{k,0},\ldots,\Psi_{k,k})\in\mathfrak B_k(\mathscr A)$, 
and every $0\leq i\leq k$, one has
$\|\Psi_{k,i}\| \leq c\rho^{k-i}.$
This condition is equivalent to uniform exponential stability of the
origin for the homogeneous linear difference inclusion
$x^+\in\{Ax:A\in\mathscr A\}$.

\begin{remark}
    For a finite matrix family, uniform exponential stability is
    equivalent to its joint spectral radius being strictly smaller than
    one~\cite{Wirth2002JSR}. In particular, if
    $\mathscr A=\operatorname{conv}\{A_1,\ldots,A_s\}$, then uniform
    exponential stability of $\mathscr A$ is equivalent to the joint
    spectral radius of the vertex family $\{\overline A_1,\ldots,\overline A_s\}$ 
    being strictly smaller than one. The existence of a common quadratic
    Lyapunov function is sufficient, but generally not necessary, for
    uniform exponential stability~\cite{Shorten2007Stability}. 
    In \cite{DeyBhasin2024Computation}, the closed-loop matrices are
    described as pointwise Schur stable under a common feedback gain,
    whereas a common quadratic Lyapunov matrix is subsequently employed.
    The latter provides a sufficient condition for uniform exponential
    stability, while pointwise Schur stability alone does not guarantee
    uniform exponential stability under arbitrary parameter variations.
\end{remark}

For each $k\geq1$, define the $k$-step
disturbance-reachable set from the origin by
\begin{equation}
    \mc R_k
    :=
    \left\{
        \sum_{i=0}^{k-1}
        \Psi_{k,i+1}w_i:
        \boldsymbol\Psi_k\in\mathfrak B_k(\mathscr A),\
        w_0,\ldots,w_{k-1}\in\mc W
    \right\}
\end{equation}

\begin{corollary}
    The origin belongs to the maximal robust positively invariant set if
    and only if the origin satisfies the state constraint and every $k$-step
    disturbance-reachable set from the origin remains admissible with respect to the state constraint, i.e.,
    \begin{equation}
        0\in\mc M
        \quad\Longleftrightarrow\quad
        0\in\mc X
        \ \text{and}\
        \mc R_k\subseteq\mc X,
        \quad  k\geq1.
    \end{equation}
\end{corollary}

\begin{proof}
    By the characterization~\eqref{eq:comp:xk_finite_horizon},
    $0\in\mc X_k$ if and only if
    $0\in\mc X$ and
    $\mc R_j\subseteq\mc X$ for every $j=1,\ldots,k$.
    Since
    $\mc M=\bigcap_{k\in\N}\mc X_k$,
    the result follows.
\end{proof}

Define the limiting disturbance-reachable set by
\begin{equation}
\label{eq:nonempty:R}
    \mc R
    :=
    \bigcap_{\ell\geq 1}
    \operatorname{cl}\left(
    \bigcup_{k\geq \ell}\mc R_k
    \right).
\end{equation}
Equivalently, $\mc R$ is the Kuratowski upper limit of the
sequence $\{\mc R_k\}_{k\geq 1}$.

\begin{remark}
    If $0\in\mc W$, then
    $\mc R_k\subseteq\mc R_{k+1}$ for every $k\geq1$. Indeed, any
    element of $\mc R_k$ can be realized in $\mc R_{k+1}$ by prepending
    a zero disturbance and shifting the original matrix and disturbance
    sequences by one time step. Hence, in this case,
    \eqref{eq:nonempty:R} reduces to 
    $\mc R = \operatorname{cl} \left( \bigcup_{k\geq1}\mc R_k \right).$
\end{remark}

The following proposition establishes the fundamental properties of
the set $\mc R$. 

\begin{proposition}
    \label{prop:minimal_rpi}
    Suppose that the nonempty matrix family $\mathscr A$ is uniformly
    exponentially stable and that $\mc W$ is nonempty and bounded.
    Then the set $\mc R$ defined in~\eqref{eq:nonempty:R} is nonempty
    and compact and satisfies $\mc F(\mc R)\subseteq\mc R.$
    Moreover, $\mc R$ is contained in every nonempty closed set
    $\mc S\subseteq\R^n$ satisfying
    $\mc F(\mc S)\subseteq\mc S$.
\end{proposition}

\begin{proof}
    The proof can be found in Appendix~\ref{app:minimal_rpi}.
\end{proof}

Proposition~\ref{prop:minimal_rpi} therefore identifies $\mc R$ as the
least, with respect to set inclusion, nonempty closed robust positively
invariant set for the dynamics $x^+\in\mc F(x)$.
Following standard set invariance terminology
\cite{BlanchiniMiani2015}, we also refer to $\mc R$ as the minimal closed
robust positively invariant set. Related existence and outer-approximation results for the minimal
robust positively invariant set of polytopic linear difference
inclusions were established in
\cite{KouramasEtAl2005MinimalRPI} under common quadratic stability and
compact convex disturbance sets containing the origin in their
interior.  Proposition~\ref{prop:minimal_rpi} establishes the
corresponding minimality property for general uniformly exponentially
stable matrix families and nonempty bounded disturbance sets. This
property yields the following necessary and sufficient condition for
nonemptiness of $\mc M$. 

\begin{theorem}
\label{thm:nonempty}
    Suppose that the nonempty matrix family $\mathscr A$ is uniformly
    exponentially stable and that $\mc W$ is nonempty and bounded. Let
    $\mc X\subseteq\R^n$ be closed. Then the maximal robust positively
    invariant set $\mc M$ is nonempty if and only if
    \begin{equation}
    \label{eq:nonempty_condition}
        \mc R\subseteq\mc X.
    \end{equation}
\end{theorem}

\begin{proof}
Suppose first that $\mc R\subseteq\mc X$. By
Proposition~\ref{prop:minimal_rpi}, the set $\mc R$ is nonempty and
satisfies $\mc F(\mc R)\subseteq\mc R.$
Hence, $\mc R$ is a nonempty robust positively invariant subset of
$\mc X$, and therefore $\mc M$ is nonempty.
Conversely, suppose that $\mc M$ is nonempty. Since $\mc X$ is closed,
Corollary~\ref{cor:closed_M} implies that $\mc M$ is closed. Moreover,
$\mc F(\mc M)\subseteq\mc M$. The minimality property in
Proposition~\ref{prop:minimal_rpi} therefore gives
$\mc R\subseteq\mc M\subseteq\mc X.$
\end{proof}

Theorem~\ref{thm:nonempty} shows that nonemptiness of $\mc M$ is
equivalent to admissibility of the minimal closed robust positively
invariant set $\mc R$ with respect to the state constraint $\mc X$.
Building on the minimal robust positively invariant set results for
polytopic linear difference inclusions in
\cite{KouramasEtAl2005MinimalRPI},
\cite{DeyBhasin2024Computation} derives a related criterion under
polytopic uncertainty and polyhedral state constraints by
testing containment of the convex hull of the minimal set. The present
result instead allows an arbitrary closed state constraint set and is
stated directly in terms of $\mc R$.
Although condition~\eqref{eq:nonempty_condition} is exact, it may be
difficult to verify directly because $\mc R$ is defined through the
asymptotic behavior of the disturbance reachable sets
$\{\mc R_k\}_{k\geq 1}$. The following corollary provides a simpler
sufficient condition based on a common contractive set.

\begin{corollary}
\label{cor:nonempty_contraction}
    Let $\mc L\subseteq\R^n$ be a nonempty closed convex set containing
    the origin. Suppose that there exist $\lambda\in[0,1)$ and
    $\nu\geq0$ such that
    \begin{equation*}
        A\mc L\subseteq\lambda\mc L
        \quad\text{for every }A\in\mathscr A,
        \qquad
        \mc W\subseteq\nu\mc L.
    \end{equation*}
    If $\nu/(1-\lambda)\mc L\subseteq\mc X$, then
    $\nu/(1-\lambda)\mc L$ is robust positively invariant and satisfies
    \[
        \mc R
        \subseteq
        \nu/(1-\lambda)\mc L
        \subseteq
        \mc M
        \subseteq
        \mc X.
    \]
    Consequently, $\mc M$ is nonempty.
\end{corollary}

\begin{proof}
    Set $\gamma:=\nu/(1-\lambda)$. By the convexity of $\mc L$ and the
    fact that $0\in\mc L$, we have 
    $\mc F(\gamma\mc L) \subseteq \lambda\gamma\mc L+\nu\mc L
    = (\lambda\gamma+\nu)\mc L = \gamma\mc L.$
    Together with the assumed inclusion
    $\gamma\mc L\subseteq\mc X$, this gives
    $\mc F(\gamma\mc L) \subseteq \gamma\mc L \subseteq \mc X.$
    Hence, $\gamma\mc L$ is robust positively invariant. By the
    maximality of $\mc M$, we have 
    $\gamma\mc L\subseteq\mc M\subseteq\mc X.$
    Moreover, repeated application of
    $A\mc L\subseteq\lambda\mc L$ for every $A\in\mathscr A$ and
    $\mc W\subseteq\nu\mc L$ gives, for every $k\geq1$,
    $\mc R_k \subseteq \nu\left(\sum_{j=0}^{k-1}\lambda^j\right)\mc L
    \subseteq \gamma\mc L.$
    Since $\gamma\mc L$ is closed, the definition of $\mc R$ yields
    $\mc R\subseteq\gamma\mc L$. Therefore, we have
    $\mc R \subseteq \nu/(1-\lambda)\mc L \subseteq \mc M \subseteq \mc X.$
    In particular, $\mc M$ is nonempty.
\end{proof}

The exact nonemptiness condition $\mc R\subseteq\mc X$
does not, by itself, guarantee finite determination.
Finite-determination conditions and determination index
bounds for linear time-invariant systems with additive
disturbances have been studied in
\cite{Rakovic2004,Nguyen2011}. 
Lyapunov-based finite-determination bounds have also
been established for disturbance-free linear
time-invariant systems \cite{Ossareh2024complexity}.
The following theorem establishes a finite-determination
criterion for linear difference inclusions under uniform
exponential stability and provides an explicit upper
bound on the determination index.

\begin{theorem}
\label{thm:finite_strict_nonempty}
    Suppose that the nonempty matrix family $\mathscr A$ is
    uniformly exponentially stable, with constants
    $c\geq1$ and $\rho\in[0,1)$ as in
    \eqref{eq:uniform_exponential_stability}.
    Consider the sequence $\{\mc X_k\}_{k\in\N}$
    generated by any one of the equivalent iterations
    \eqref{eq:equiv_iteration_1}--\eqref{eq:equiv_iteration_3},
    with $\mc X_0=\mc X$.
    Suppose that there exist a nonempty closed set
    $\mc S\subseteq\R^n$ and $\varepsilon>0$ such that
    \begin{equation}
    \label{eq:finite_invariant_margin}
        \mc F(\mc S)\subseteq\mc S,
        \qquad
        \mc S+\varepsilon\mathbb B\subseteq\mc X,
    \end{equation}
    where $\mathbb B$ denotes the closed unit ball associated
    with the underlying vector norm $\|\cdot\|$.
    Suppose also that $\mc X_{\bar k}$ is bounded for some
    $\bar k\in\N$, and define
    \begin{equation}
    \label{eq:finite_beta}
        \beta:=
        \sup_{x\in\mc X_{\bar k}}
        \operatorname{dist}(x,\mc S).
    \end{equation}
    Then $0\leq\beta<\infty$, and the maximal robust positively
    invariant set $\mc M$ is nonempty and finitely determined.
    More precisely, for any integer $q\geq\bar k+1$
    satisfying
    \begin{equation}
    \label{eq:finite_explicit_condition}
        c\rho^q\beta\leq\varepsilon,
    \end{equation}
    it holds that $\mc M=\mc X_{q-1}.$
    In particular, if $\rho\in(0,1)$ and $\beta>0$,
    one may choose
    \begin{equation}
    \label{eq:finite_explicit_bound}
        q=
        \max\left\{
        \bar k+1,\,
        \left\lceil
        \frac{\ln(c\beta/\varepsilon)}
             {\ln(1/\rho)}
        \right\rceil
        \right\}.
    \end{equation}
    If $\rho=0$ or $\beta=0$, one may take
    $q=\bar k+1$.
\end{theorem}

\begin{proof}
    The proof can be found in
    Appendix~\ref{app:finite_strict_nonempty}.
\end{proof}

Under the assumptions of Proposition~\ref{prop:minimal_rpi},
the set $\mc R$ is nonempty, compact, and robust positively
invariant. Thus, if
$\mc R\subseteq\operatorname{int}(\mc X)$ and
$\mc X_{\bar k}$ is bounded for some $\bar k\in\N$,
Theorem~\ref{thm:finite_strict_nonempty} applies
with $\mc S=\mc R$. In particular, when $\mc X$ is
compact, the boundedness condition is automatically
satisfied with $\bar k=0$.

\begin{remark}
    Theorem~\ref{thm:finite_strict_nonempty} does not
    require an explicit characterization of $\mc R$.
    Instead, any nonempty closed robust positively
    invariant set satisfying the containment condition
    in \eqref{eq:finite_invariant_margin} can be used
    to certify finite determination and obtain an
    explicit bound on the determination index.
    For example, Corollary~\ref{cor:nonempty_contraction}
    can be applied to the ellipsoid
    $\mc L:=\{x\in\R^n:x^\top Px\leq1\}$
    with $P=P^\top\succ0$.
    Suppose that, for some $\lambda\in(0,1)$
    and $\nu\geq0$,
    \[
        A^\top PA\preceq\lambda^2P
        \quad\text{for every }A\in\mathscr A,
        \qquad
        \mc W\subseteq\nu\mc L.
    \]
    Then $\mc S:=\gamma\mc L$, where
    $\gamma:=\nu/(1-\lambda)$, is robust positively
    invariant whenever $\mc S\subseteq\mc X$.
    Choosing the weighted Euclidean norm
    $\|x\|_P:=\sqrt{x^\top Px}$, whose closed
    unit ball is $\mc L$, gives
    $\mc S+\varepsilon\mc L
        =(\gamma+\varepsilon)\mc L.$
    Hence, if $(\gamma+\varepsilon)\mc L\subseteq\mc X$
    for some $\varepsilon>0$ and $\mc X_{\bar k}$
    is bounded for some $\bar k\in\N$, then
    Theorem~\ref{thm:finite_strict_nonempty}
    applies with $c=1$ and $\rho=\lambda$.
\end{remark}

\section{Algorithmic Computation}
\label{sec:algorithmic_computation}
The results of Section~\ref{sec:TMRPIS} lead directly to three algorithms 
for generating the common set-iteration sequence. This section first
formulates these algorithms in terms of general set operations and then
develops exact polyhedral implementations, together with their
computational complexity and implementation tradeoffs.

\subsection{Algorithms and Computational Primitives}

Algorithm~\ref{alg:standard_self_restricted} combines the standard and
self-restricted iterations. At iteration $k$, it first constructs the
intermediate predecessor set
$\overline{\mc X}_{k+1}
    =
    \bigcap_{A\in\mathscr A}
    A^{-1}(\mc X_k-\mc W).$
This requires forming the Pontryagin difference $\mc X_k-\mc W$,
computing its preimage under each admissible matrix, and intersecting
the resulting preimages over $\mathscr A$.
The inclusion
$\mc X_k\subseteq\overline{\mc X}_{k+1}$ is used as the exact
stopping test. By Theorem~\ref{thm:finite_determination}, if this
condition holds, then $\mc M=\mc X_k$. Otherwise, the next iterate is
formed as
$\mc X_{k+1}
    =
    \overline{\mc X}_{k+1}\bigcap\mc C_k,$
where $\mc C_k=\mc X$ for the standard iteration and
$\mc C_k=\mc X_k$ for the self-restricted iteration. Because
$\mc X_k\subseteq\mc C_k$ in both cases, intersecting with $\mc C_k$
does not affect the stopping test, which may therefore be evaluated
before $\mc X_{k+1}$ is constructed.

\begin{algorithm}[!t]
\caption{Standard and Self-Restricted Set Iterations}
\label{alg:standard_self_restricted}
\begin{algorithmic}[1]
    \Require $\mathscr A$, $\mc W$, $\mc X$, and
    $\mathrm{mode}\in
    \{\mathrm{standard},\mathrm{self\text{-}restricted}\}$
    \Ensure The exact set $\mc M$, if finite determination or emptiness
    is detected
    \State $k\gets0$
    \State $\mc X_0\gets\mc X$
    \Loop
        \State Compute
        $\displaystyle
        \overline{\mc X}_{k+1}
        \gets
        \bigcap_{A\in\mathscr A}
        A^{-1}(\mc X_k-\mc W)$
        \If{$\mc X_k\subseteq\overline{\mc X}_{k+1}$}
            \State \Return $\mc M=\mc X_k$
        \EndIf
        \If{$\mathrm{mode}=\mathrm{standard}$}
            \State $\mc C_k\gets\mc X$
        \ElsIf{$\mathrm{mode}=\mathrm{self\text{-}restricted}$}
            \State $\mc C_k\gets\mc X_k$
        \EndIf
        \State Compute
        $\mc X_{k+1}
        \gets
        \overline{\mc X}_{k+1}\bigcap\mc C_k$
        \If{$\mc X_{k+1}=\emptyset$}
            \State \Return $\mc M=\emptyset$
        \EndIf
        \State $k\gets k+1$
    \EndLoop
\end{algorithmic}
\end{algorithm}

\begin{algorithm}[!t]
\caption{Incremental Set Iteration}
\label{alg:incremental_set_iteration}
\begin{algorithmic}[1]
    \Require $\mathscr A$, $\mc W$, and $\mc X$
    \Ensure The exact set $\mc M$, if finite determination or emptiness
    is detected
    \State $k\gets0$
    \State $\mc X_0\gets\mc X$
    \Loop
        \State Compute $\mc Y_{k+1}$ from~\eqref{eq:comp:Y_k}
        \If{$\mc X_k\subseteq\mc Y_{k+1}$}
            \State \Return $\mc M=\mc X_k$
        \EndIf
        \State Compute
        $\mc X_{k+1}\gets\mc X_k\bigcap\mc Y_{k+1}$
        \If{$\mc X_{k+1}=\emptyset$}
            \State \Return $\mc M=\emptyset$
        \EndIf
        \State $k\gets k+1$
    \EndLoop
\end{algorithmic}
\end{algorithm}

Algorithm~\ref{alg:incremental_set_iteration} instead constructs the
new constraint set $\mc Y_{k+1}$ appearing in the incremental
iteration. By~\eqref{eq:comp:Y_k}, this involves, for every
matrix-product tuple
$\boldsymbol{\Psi}_{k+1}\in\mathfrak B_{k+1}(\mathscr A)$, forming the
accumulated disturbance set
$\sum_{i=0}^{k}\Psi_{k+1,i+1}\mc W,$
taking its Pontryagin difference from $\mc X$, and computing the
preimage under the full product $\Psi_{k+1,0}$. Intersecting these
preimages over all admissible matrix-product tuples gives
$\mc Y_{k+1}$. The inclusion $\mc X_k\subseteq\mc Y_{k+1}$ is the 
corresponding exact stopping test. If it holds, 
Theorem~\ref{thm:finite_determination}
gives $\mc M=\mc X_k$. Otherwise, the newly imposed constraint is
appended through
$\mc X_{k+1}=\mc X_k\bigcap\mc Y_{k+1}$.


By Proposition~\ref{prop:equivalent_iterations}, the two algorithms
generate the same decreasing sequence
$\{\mc X_k\}_{k\in\N}$. If an iterate becomes empty, then
Corollary~\ref{cor:MRPI} gives
$\mc M\subseteq\mc X_{k+1}=\emptyset$, and hence
$\mc M=\emptyset$. If neither an exact stopping condition nor emptiness
is detected, the algorithms continue to generate the set-iteration
sequence; finite termination is not guaranteed in general. In a numerical 
implementation, a maximum
iteration index $k_{\max}$ may therefore be imposed. If this limit is
reached before termination, the current iterate $\mc X_{k_{\max}}$
provides an outer approximation of $\mc M$, but finite determination
has not been certified.
The emptiness tests in
Algorithms~\ref{alg:standard_self_restricted} and
\ref{alg:incremental_set_iteration} may be omitted when nonemptiness
has been certified a priori. Under the assumptions of
Theorem~\ref{thm:nonempty}, such a certificate is provided by
$\mc R\subseteq\mc X$, or by a certified outer approximation
$\mc R^{\mathrm{out}}$ satisfying
$\mc R\subseteq\mc R^{\mathrm{out}}\subseteq\mc X$.

In summary, Algorithms~\ref{alg:standard_self_restricted} and
\ref{alg:incremental_set_iteration} rely on the same basic
set-theoretic operations: Pontryagin differences, linear preimages,
intersections, inclusion tests, and emptiness tests. The former
constructs the next iterate from the one-step robust predecessor of
$\mc X_k$, whereas the latter directly constructs $\mc Y_{k+1}$ and
updates the current iterate according to
$\mc X_{k+1}=\mc X_k\bigcap\mc Y_{k+1}$. Their relative computational
cost depends on the representations of $\mathscr A$, $\mc W$ and the
iterates.

\subsection{Polyhedral Implementations for Finite and Polytopic Matrix Families}
\label{subsec:polyhedral_polytopic}
We now specialize the set-theoretic algorithms developed above to
polyhedral state constraints and finite or polytopic matrix families.
Throughout this subsection, $\mc W$ is assumed to be nonempty and
bounded, and all support-function evaluations required by the
algorithms are assumed to be performed exactly. Under these
assumptions, everyfinite iterate admits an exact polyhedral
representation.  
Let the nonempty polyhedral state constraint set be represented as
\begin{equation}
\label{eq:poly:X}
    \mc X
    =
    \{x\in\R^n:H_{\mc X}x\leq h_{\mc X}\},
\end{equation}
where
$H_{\mc X}\in\R^{m_{\mc X}\times n}$ and
$h_{\mc X}\in\R^{m_{\mc X}}$. Let
$\mathscr A_{\rm v} := \{\overline A_1,\ldots,\overline A_s\}$
denote a finite family of matrices. We consider either the finite
matrix family
\begin{equation}
\label{eq:poly:finite_A}
    \mathscr A
    =
    \mathscr A_{\rm v}
    =
    \{\overline A_1,\ldots,\overline A_s\},
\end{equation}
or the polytopic matrix family
\begin{equation}
\label{eq:poly:polytopic_A}
    \mathscr A
    =
    \operatorname{conv}(\mathscr A_{\rm v})
    =
    \operatorname{conv}
    \{\overline A_1,\ldots,\overline A_s\}.
\end{equation}

The following result establishes the polyhedrality of the finite
iterates and shows that, in the polytopic case, the full matrix family
can be replaced exactly by the finite family $\mathscr A_{\rm v}$.
Consequently, both cases admit a common finite-family implementation.

\begin{corollary}
\label{cor:polyhedral_iterates}
    Let $\mc X$ be given by~\eqref{eq:poly:X}, suppose that $\mc W$ is nonempty and
    bounded, and suppose that $\mathscr A$ is given by either
    \eqref{eq:poly:finite_A} or~\eqref{eq:poly:polytopic_A}. Consider
    the sequence $\{\mc X_k\}_{k\in\N}$ generated by any one of the
    equivalent iterations
    \eqref{eq:equiv_iteration_1}--\eqref{eq:equiv_iteration_3}, with
    $\mc X_0=\mc X$, and the sequence
    $\{\mc Y_k\}_{k\in\N}$ defined by $\mc Y_0=\mc X$
    and~\eqref{eq:comp:Y_k}. Consider also the intermediate predecessor set sequence 
    $\{\overline{\mc X}_{k}\}_{k\ge1}$ defined in~\eqref{eq:intermediate_predecessor}. 
    Then:
    \begin{enumerate}
    \renewcommand{\labelenumi}{(\roman{enumi})}
        \item the sets $\mc X_k$ and $\mc Y_k$ are polyhedral for every
        $k\in\N$. Consequently, if $\mc M$ is finitely determined,
        then $\mc M$ is polyhedral;

        \item in the polytopic case~\eqref{eq:poly:polytopic_A}, for every
        $k\in\N$, the
        intermediate predecessor set admits the exact finite-family
        representation
        \begin{equation}
        \label{eq:poly:vertex_predecessor}
            \overline{\mc X}_{k+1}
            =
            \bigcap_{A\in\mathscr A}
            A^{-1}(\mc X_k-\mc W)
            =
            \bigcap_{i=1}^{s}
            \overline A_i^{-1}(\mc X_k-\mc W),
        \end{equation}
        and, for every $k\geq1$,
        \begin{align}
            \mc Y_k
            &=
            \bigcap_{\boldsymbol{\Psi}_k
            \in\mathfrak B_k(\mathscr A)}
            \Psi_{k,0}^{-1}
            \left(
                \mc X
                -
                \sum_{i=0}^{k-1}
                \Psi_{k,i+1}\mc W
            \right)
            \nonumber\\
            &=
            \bigcap_{\boldsymbol{\Psi}_k
            \in\mathfrak B_k(\mathscr A_{\rm v})}
            \Psi_{k,0}^{-1}
            \left(
                \mc X
                -
                \sum_{i=0}^{k-1}
                \Psi_{k,i+1}\mc W
            \right).
            \label{eq:poly:vertex_Yk}
        \end{align}
    \end{enumerate}
\end{corollary}

\begin{proof}
    The proof can be found in Appendix~\ref{app:polyhedral_iterates}.
\end{proof}

Accordingly, the algorithms below are written solely in terms of the
finite family $\mathscr A_{\rm v}$. This is immediate in the finite
case~\eqref{eq:poly:finite_A}, while in the polytopic
case~\eqref{eq:poly:polytopic_A} it follows from the exact vertex
reduction established in Corollary~\ref{cor:polyhedral_iterates}.
In what follows, $\operatorname{reduce}(H,h)$ denotes an equivalent
half-space representation obtained by deleting identified redundant
inequalities from the system $Hx\leq h$.

\subsubsection{Standard and Self-Restricted Implementations}
Suppose that, at iteration $k$, the current iterate is represented as
\begin{equation}
    \mc X_k
    =
    \{x\in\R^n:H_kx\leq h_k\},
\end{equation}
where $H_k\in\R^{m_k\times n}$ and $h_k\in\R^{m_k}$. Since
$\mc W$ is nonempty and bounded, its support function is finite-valued. 
The Pontryagin difference therefore admits the representation
\begin{equation}
\label{eq:poly:pontryagin_difference}
    \mc X_k-\mc W
    =
    \left\{
        z\in\R^n:
        H_kz
        \leq
        h_k-\support{\mc W}{H_k}
    \right\}.
\end{equation}
Let $d_k:=\support{\mc W}{H_k}$. By
Corollary~\ref{cor:polyhedral_iterates}, the intermediate predecessor
set given in \eqref{eq:poly:vertex_predecessor} can be written as
\begin{equation}
\label{eq:poly:intermediate_representation}
    \overline{\mc X}_{k+1}
    =
    \{x\in\R^n:
    \overline H_{k+1}x\leq\overline h_{k+1}\},
\end{equation}
where
\begin{equation}
\label{eq:poly:intermediate_matrices}
    \overline H_{k+1}
    =
    \begin{bmatrix}
        H_k\overline A_1\\
        \vdots\\
        H_k\overline A_s
    \end{bmatrix},
    \qquad
    \overline h_{k+1}
    =
    \mathbf 1_s\otimes(h_k-d_k).
\end{equation}
Here, $\mathbf 1_s\in\R^s$ denotes the vector of ones and $\otimes$
denotes the Kronecker product.

The exact stopping test
$\mc X_k\subseteq\overline{\mc X}_{k+1}$ can be checked through the
componentwise condition
\begin{equation}
\label{eq:poly:standard_stopping}
    \support{\mc X_k}{\overline H_{k+1}}
    \leq
    \overline h_{k+1}.
\end{equation}
If this condition holds, then
Theorem~\ref{thm:finite_determination} gives
$\mc M=\mc X_k$. Otherwise, the standard iteration constructs
$\mc X_{k+1}=\overline{\mc X}_{k+1}\bigcap\mc X$ by appending
$(H_{\mc X},h_{\mc X})$ to the half-space representation of
$\overline{\mc X}_{k+1}$. The resulting polyhedral implementation of the standard iteration is
summarized in Algorithm~\ref{alg:poly_standard}.

\begin{algorithm}[!t]
\caption{Standard Iteration: Polyhedral Implementation}
\label{alg:poly_standard}
\begin{algorithmic}[1]
\Require
$\mathscr A_{\rm v}
=\{\overline A_1,\ldots,\overline A_s\}$,
$\mc W$, $H_{\mc X}$, and $h_{\mc X}$
\Ensure A half-space representation of $\mc M$, or the certificate
$\mc M=\emptyset$, if detected
\State $k\gets0$
\State $(H_0,h_0)\gets(H_{\mc X},h_{\mc X})$
\Loop
    \State $\mc X_k\gets\{x:H_kx\leq h_k\}$
    \State $d_k\gets\support{\mc W}{H_k}$
    \State
    $\displaystyle
    \overline H_{k+1}
    \gets
    \begin{bmatrix}
        H_k\overline A_1\\
        \vdots\\
        H_k\overline A_s
    \end{bmatrix}$
    \State
    $\overline h_{k+1}
    \gets
    \mathbf 1_s\otimes(h_k-d_k)$
    \State
    $r\gets\support{\mc X_k}{\overline H_{k+1}}$
    \If{$r\leq\overline h_{k+1}$}
        \State \Return $\mc M=\mc X_k$
    \EndIf
    \State
    $\displaystyle
    \widehat H_{k+1}
    \gets
    \begin{bmatrix}
        \overline H_{k+1}\\
        H_{\mc X}
    \end{bmatrix}$
    \State
    $\displaystyle
    \widehat h_{k+1}
    \gets
    \begin{bmatrix}
        \overline h_{k+1}\\
        h_{\mc X}
    \end{bmatrix}$
    \If{$\{x:\widehat H_{k+1}x\leq\widehat h_{k+1}\}
        =\emptyset$}
        \State \Return $\mc M=\emptyset$
    \EndIf
    \If{redundancy removal is enabled}
        \State
        $(H_{k+1},h_{k+1})
        \gets
        \operatorname{reduce}
        (\widehat H_{k+1},\widehat h_{k+1})$
    \Else
        \State
        $(H_{k+1},h_{k+1})
        \gets
        (\widehat H_{k+1},\widehat h_{k+1})$
    \EndIf
    \State $k\gets k+1$
\EndLoop
\end{algorithmic}
\end{algorithm}

The self-restricted implementation uses the same intermediate
representation and stopping test. It differs only in the construction
of the next iterate: instead of appending
$(H_{\mc X},h_{\mc X})$, it appends the current representation
$(H_k,h_k)$.

\subsubsection{Incremental Implementation}
For each $k\in\N$, let
$\Xi_{k+1}:=\{1,\ldots,s\}^{k+1}$ denote the set of matrix-index
sequences of length $k+1$. Each
$\sigma=(\sigma_0,\ldots,\sigma_k)\in\Xi_{k+1}$ determines the ordered
matrix sequence
$(\overline A_{\sigma_0},\ldots,\overline A_{\sigma_k})$
from the finite family $\mathscr A_{\rm v}$. Let
\begin{equation}
    \boldsymbol{\Psi}_{k+1}(\sigma)
    :=
    \bigl(
        \Psi_{k+1,0}(\sigma),
        \ldots,
        \Psi_{k+1,k+1}(\sigma)
    \bigr)
\end{equation}
denote the corresponding matrix-product tuple, with
\begin{equation*}
    \begin{aligned}
        \Psi_{k+1,k+1}(\sigma)
        &:=
        I,\\
        \Psi_{k+1,i}(\sigma)
        &:=
        \overline A_{\sigma_k}\cdots\overline A_{\sigma_i},
        \qquad i=0,\ldots,k.
    \end{aligned}
\end{equation*}
The state reached after $k+1$ steps from $x_0$ along this ordered
matrix sequence
$(\overline A_{\sigma_0},\ldots,\overline A_{\sigma_k})$ is
\begin{equation*}
    x_{k+1}
    =
    \Psi_{k+1,0}(\sigma)x_0
    +
    \sum_{i=0}^{k}
    \Psi_{k+1,i+1}(\sigma)w_i.
\end{equation*}

For a fixed $\sigma\in\Xi_{k+1}$, requiring
$x_{k+1}\in\mc X$ for every
$w_0,\ldots,w_k\in\mc W$ is equivalent to
\begin{equation}
    \label{eq:poly:sequence_constraint}
    H_{\mc X}\Psi_{k+1,0}(\sigma)x_0
    \leq
    h_{\mc X}
    -
    \sum_{i=0}^{k}
    \support{\mc W}{
        H_{\mc X}\Psi_{k+1,i+1}(\sigma)
    }
\end{equation}
where the inequality is understood componentwise.

The resulting polyhedral implementation of the incremental iteration is
summarized in Algorithm~\ref{alg:poly_incremental}. In Algorithm~\ref{alg:poly_incremental}, the two sides of
\eqref{eq:poly:sequence_constraint} are evaluated by processing the
matrices in reverse time order. For each $\sigma\in\Xi_{k+1}$, initialize
$Q_\sigma:=H_{\mc X}$ and $\Delta_\sigma:=\mathbf{0}_{m_{\mc X}}$, where $\mathbf{0}_{m_{\mc X}}\in\R^{m_{\mc X}}$ denotes the vector of zeros. At each step, the 
vector $\support{\mc W}{Q_\sigma}$ is added to
$\Delta_\sigma$, after which $Q_\sigma$ is right-multiplied by the
corresponding matrix $\overline A_{\sigma_i}$. After all $k+1$
matrices have been processed, we have 
\begin{equation}
    \begin{aligned}
        Q_\sigma
        &=
        H_{\mc X}\Psi_{k+1,0}(\sigma),\\
        \Delta_\sigma
        &=
        \sum_{i=0}^{k}
        \support{\mc W}{
            H_{\mc X}\Psi_{k+1,i+1}(\sigma)
        }.
    \end{aligned}
\end{equation}
Hence, the admissible initial states associated with $\sigma$ satisfy
\begin{equation}
    Q_\sigma x_0
    \leq
    h_{\mc X}-\Delta_\sigma.
\end{equation}
Stacking these inequalities over all
$\sigma\in\Xi_{k+1}$ yields the half-space representation
\begin{equation}
\label{eq:poly:Yk_halfspace}
    \mc Y_{k+1}
    =
    \{x\in\R^n:G_{k+1}x\leq g_{k+1}\}.
\end{equation}

The exact incremental stopping test
$\mc X_k\subseteq\mc Y_{k+1}$ is evaluated through
\begin{equation}
\label{eq:poly:incremental_stopping}
    \support{\mc X_k}{G_{k+1}}
    \leq
    g_{k+1},
\end{equation}
where the inequality is understood componentwise. If the test holds,
Theorem~\ref{thm:finite_determination} gives
$\mc M=\mc X_k$. Otherwise, the next iterate is formed as
$\mc X_{k+1}=\mc X_k\bigcap\mc Y_{k+1}$.

\begin{algorithm}[!t]
\caption{Incremental Iteration: Polyhedral Implementation}
\label{alg:poly_incremental}
\begin{algorithmic}[1]
\Require
$\mathscr A_{\rm v}
=\{\overline A_1,\ldots,\overline A_s\}$,
$\mc W$, $H_{\mc X}$, and $h_{\mc X}$
\Ensure A half-space representation of $\mc M$, or the certificate
$\mc M=\emptyset$, if detected
\State $k\gets0$
\State $(H_0,h_0)\gets(H_{\mc X},h_{\mc X})$
\Loop
    \State $\mc X_k\gets\{x:H_kx\leq h_k\}$
    \State Initialize $(G_{k+1},g_{k+1})$ as an empty pair
    \State $\Xi_{k+1}\gets\{1,\ldots,s\}^{k+1}$
    \For{each
        $\sigma=(\sigma_0,\ldots,\sigma_k)\in\Xi_{k+1}$}
        \State $Q_\sigma\gets H_{\mc X}$
        \State $\Delta_\sigma\gets\mathbf{0}_{m_{\mc X}}$
        \For{$i=k,k-1,\ldots,0$}
            \State
            $\Delta_\sigma\gets\Delta_\sigma+\support{\mc W}{Q_\sigma}$
            \State
            $Q_\sigma\gets Q_\sigma\overline A_{\sigma_i}$
        \EndFor
        \State
        $\displaystyle
        G_{k+1}
        \gets
        \begin{bmatrix}
            G_{k+1}\\
            Q_\sigma
        \end{bmatrix}$
        \State
        $\displaystyle
        g_{k+1}
        \gets
        \begin{bmatrix}
            g_{k+1}\\
            h_{\mc X}-\Delta_\sigma
        \end{bmatrix}$
    \EndFor
    \State
    $\mc Y_{k+1}
    \gets
    \{x:G_{k+1}x\leq g_{k+1}\}$
    \State
    $r\gets\support{\mc X_k}{G_{k+1}}$
    \If{$r\leq g_{k+1}$}
        \State \Return $\mc M=\mc X_k$
    \EndIf
    \State
    $\displaystyle
    \widehat H_{k+1}
    \gets
    \begin{bmatrix}
        H_k\\
        G_{k+1}
    \end{bmatrix}$
    \State
    $\displaystyle
    \widehat h_{k+1}
    \gets
    \begin{bmatrix}
        h_k\\
        g_{k+1}
    \end{bmatrix}$
    \If{$\{x:\widehat H_{k+1}x\leq\widehat h_{k+1}\}
        =\emptyset$}
        \State \Return $\mc M=\emptyset$
    \EndIf
    \If{redundancy removal is enabled}
        \State
        $(H_{k+1},h_{k+1})
        \gets
        \operatorname{reduce}
        (\widehat H_{k+1},\widehat h_{k+1})$
    \Else
        \State
        $(H_{k+1},h_{k+1})
        \gets
        (\widehat H_{k+1},\widehat h_{k+1})$
    \EndIf
    \State $k\gets k+1$
\EndLoop
\end{algorithmic}
\end{algorithm}

Algorithm~\ref{alg:poly_standard}, its self-restricted variant, and
Algorithm~\ref{alg:poly_incremental} test the unreduced half-space
representation of $\mc X_{k+1}$ for emptiness before optional
redundancy removal. If the represented polyhedron is empty, the algorithm 
returns
$\mc M=\emptyset$. Otherwise, the representation may be retained or
simplified without changing $\mc X_{k+1}$. Removing redundant
inequalities can reduce the cost of subsequent support-function
evaluations, inclusion tests, and feasibility checks. The associated
computational costs are discussed next.

\subsubsection{Complexity and Implementation Tradeoffs}
\label{subsubsec:polyhedral_complexity}
We compare the direct polyhedral implementations in terms of generated
inequalities, dense matrix arithmetic, optimization subproblems, and
storage.  Recall that
$m_{\mc X}$ and $m_k$ are the numbers of inequalities in the
half-space representations of $\mc X$ and $\mc X_k$, respectively. 
Table~\ref{tab:polyhedral_operation_counts} summarizes the operation
counts before the optional reduction of the newly constructed
representation at iteration $k$. For a matrix $H$ with
$q$ rows, the row-wise quantity $\support{\mc S}{H}$ is counted as
$q$ scalar support-function evaluations.

\begin{table}[htbp!]
\centering
\caption{Numbers of inequalities in the unreduced representation
of $\mc X_{k+1}$ and scalar support-function evaluations at
iteration $k$.}
\label{tab:polyhedral_operation_counts}
\small
\setlength{\tabcolsep}{3pt}
\renewcommand{\arraystretch}{1.12}
    \begin{tabular}{@{}lccc@{}}
        \hline
        Method
        & \# inequalities
        & \# $\support{\mc W}{\cdot}$
        & \# $\support{\mc X_k}{\cdot}$
        \\
        \hline
        Standard
        & $s m_k+m_{\mc X}$
        & $m_k$
        & $s m_k$
        \\
        Self-restricted
        & $(s+1)m_k$
        & $m_k$
        & $s m_k$
        \\
        Incremental
        & $m_k+m_{\mc X}s^{k+1}$
        & $(k+1)m_{\mc X}s^{k+1}$
        & $m_{\mc X}s^{k+1}$
        \\
        \hline
    \end{tabular}
\end{table}

For dense matrices, forming the products
$H_k\overline A_i$, $i=1,\ldots,s$, in
\eqref{eq:poly:intermediate_matrices} requires
$\mathcal O(s m_k n^2)$ arithmetic operations in the standard and
self-restricted implementations. In the incremental implementation,
the update
$Q_\sigma\gets Q_\sigma\overline A_{\sigma_i}$ is performed $k+1$
times for each of the $s^{k+1}$ matrix-index sequences. Since
$Q_\sigma\in\R^{m_{\mc X}\times n}$, the corresponding cost is
$\mathcal O\bigl( (k+1)m_{\mc X}s^{k+1}n^2 \bigr).$
Both matrix-arithmetic costs are quadratic in the
state dimension $n$ and linear in the number of
propagated constraint normals. For the standard and
self-restricted implementations, dependence on $k$ enters through
$m_k$. For the incremental implementation, the factor $s^{k+1}$
causes exponential growth with $k$ when $s>1$. When $s=1$, this factor
equals one and the cost becomes
$\mathcal O((k+1)m_{\mc X}n^2)$.

The total computational cost also includes the optimization problems
used for support-function evaluations, stopping tests, and feasibility
tests. Each scalar evaluation of
$\support{\mc X_k}{\cdot}$ is a linear program with $n$ decision
variables and $m_k$ inequality constraints. The feasibility test for
the unreduced next iterate has the number of constraints shown in the
second column of Table~\ref{tab:polyhedral_operation_counts}.
The cost of evaluating $\support{\mc W}{\cdot}$ depends on the
representation of $\mc W$. For example, if $\mc W$ is given in polyhedral
half-space form, each scalar evaluation is a linear program.
Closed-form expressions are available for several common disturbance
sets.
For instance, consider
\[
    \mc W=c+D\mathbb B_p,
    \qquad
    \mathbb B_p:=\{u\in\R^r:\|u\|_p\leq1\},
\]
where \(c\in\R^n\) is the center, \(D\in\R^{n\times r}\) determines
the shape and orientation of the set, and \(p\in[1,\infty]\). Then,
for every direction \(y\in\R^n\),
\[
    \support{\mc W}{y}
    =
    c^\top y+\|D^\top y\|_{p^\ast},
\]
where \(p^\ast\) is the H\"older conjugate of \(p\).
This class includes possibly degenerate ellipsoids for $p=2$ and
affine images of boxes for $p=\infty$. In such cases, evaluating
$\support{\mc W}{\cdot}$ requires only matrix-vector and norm
computations. Thus, Table~\ref{tab:polyhedral_operation_counts}
separates the number of repeated primitive operations from the
representation-dependent cost of each operation.

To make the dependence on $k$ and $s$ explicit, suppose that no
redundancy removal is performed at any iteration. Let
$m_k^{\rm std}$, $m_k^{\rm self}$, and $m_k^{\rm inc}$ denote the
numbers of inequalities retained by the standard, self-restricted, and
incremental implementations, respectively. Since $\mc X_0=\mc X$, one has
$ m_0^{\rm std} = m_0^{\rm self} = m_0^{\rm inc} = m_{\mc X}.$
The counts in Table~\ref{tab:polyhedral_operation_counts} give
\begin{equation}
\label{eq:poly:constraint_growth}
    m_k^{\rm std}
    =
    m_k^{\rm inc}
    =
    \begin{cases}
        \displaystyle
        m_{\mc X}\frac{s^{k+1}-1}{s-1},
        & s>1,\\[2mm]
        (k+1)m_{\mc X},
        & s=1,
    \end{cases}
\end{equation}
whereas
\begin{equation}
\label{eq:poly:self_constraint_growth}
    m_k^{\rm self}
    =
    m_{\mc X}(s+1)^k.
\end{equation}
When $s>1$, all three unreduced representations grow exponentially
with $k$. The standard and incremental implementations retain the
same number of inequalities, of order $s^k$, whereas the
self-restricted representation grows with the larger base $s+1$.
Their construction costs nevertheless differ: at iteration $k$, the
incremental implementation enumerates $s^{k+1}$ matrix-index sequences
and performs $(k+1)m_{\mc X}s^{k+1}$ scalar evaluations of
$\support{\mc W}{\cdot}$. For fixed $k$, all three constraint counts
are polynomials of degree $k$ in $s$. When $s=1$, the standard and
incremental counts grow linearly with $k$, whereas the unreduced
self-restricted count grows as $2^k$.

A half-space representation with $m$ inequalities requires
$\mathcal O(mn)$ storage for its constraint matrix and
$\mathcal O(m)$ for its right-hand side. Up to constant factors, the
dominant storage requirements at iteration $k$ are therefore
\[
    \mathcal O\bigl(n(s m_k+m_{\mc X})\bigr),\quad
    \mathcal O\bigl(n(s+1)m_k\bigr),\quad
    \mathcal O\bigl(n(m_k+m_{\mc X}s^{k+1})\bigr)
\]
for the standard, self-restricted, and incremental implementations,
respectively. The incremental implementation need not store
$\Xi_{k+1}$ explicitly, since the index sequences can be generated one
at a time. It does, however, retain the
$m_{\mc X}s^{k+1}$ inequalities defining $\mc Y_{k+1}$ unless they are
screened or reduced as they are generated.

\paragraph*{Redundancy removal}
Redundancy removal trades the cost of additional linear programs
against smaller half-space representations in subsequent iterations.
If an unreduced representation contains \(M\) inequalities, complete
linear-programming-based reduction may require up to \(M\) linear
programs, each involving approximately \(M-1\) inequality constraints.
It may therefore dominate the cost of the current iteration when the
candidate representation is large.

In return, reducing the number of retained inequalities lowers memory
usage and may reduce the cost of subsequent support-function
evaluations, stopping tests, and feasibility tests. The resulting
benefit depends on both the number of remaining iterations and the way
the current representation enters the selected update. A practical
strategy is to remove duplicate inequalities, including positive rescalings, 
and other immediately identifiable redundancies at every iteration, while
performing complete reduction periodically or whenever the number of
retained inequalities exceeds a prescribed threshold. The numerical
effects of this tradeoff are examined in
Section~\ref{subsec:redundancy_timing}.

\paragraph*{Choice of implementation}
The three implementations generate the same sequence of sets but
distribute the computational burden differently. The standard
implementation provides a direct general-purpose half-space realization
of the robust predecessor recursion. It avoids explicit enumeration of
matrix-index sequences and produces a smaller unreduced candidate
representation than the self-restricted implementation.

The self-restricted implementation uses the same predecessor
construction and stopping condition, but repeatedly appends the
complete representation of the current iterate. It therefore processes
a larger unreduced candidate system. This disadvantage may be mitigated
when repeated inequalities can be identified or reused efficiently, or
when complete redundancy removal keeps the retained representation
small.

The incremental implementation separates the constraints introduced
at time \(k+1\) from those retained at earlier iterations. This
structure is useful when matrix-index sequences can be pruned,
processed in parallel, or organized in a tree that reuses common
products. Its direct implementation is most attractive when \(s\) and
the expected determination index are moderate. Reducing the
representation of \(\mc X_k\), however, does not reduce the
\(m_{\mc X}s^{k+1}\) candidate inequalities generated for
\(\mc Y_{k+1}\).

Among these direct implementations, the preferable choice therefore depends 
on the matrix-family
size, the expected determination index, the availability of
sequence-pruning or parallelization, and the redundancy-removal
strategy. The resulting runtime tradeoffs are examined numerically in
Sections~\ref{subsec:redundancy_timing} and
\ref{subsec:computational_scaling}.

\section{Numerical and Computational Study}
\label{sec:numerical_study}

This section numerically validates the equivalent set iterations and
examines the implementation tradeoffs predicted by the complexity
analysis in Section~\ref{subsubsec:polyhedral_complexity}.
All experiments were performed in MATLAB R2026b on a 64-bit Windows
laptop equipped with an Intel Core i5-1335U processor and 16~GB of
memory. All linear programs were solved using Gurobi 13.0.3 
\cite{Gurobi} with the
dual-simplex algorithm. The feasibility and optimality
tolerances were set to
$10^{-9}$, and a numerical tolerance
of $10^{-7}$ was used for set-inclusion and redundancy tests. 
While the polyhedral algorithms in Section 4.2 are exact under exact arithmetic, their numerical implementations are subject to these tolerances.
\subsection{Numerical Validation of the Three Implementations}
\label{subsec:validation_three_iterations}

We first numerically verify that the standard, self-restricted, and
incremental implementations generate the same set-iteration sequence
for the following two-dimensional polytopic linear difference
inclusion:
\begin{equation}
    \label{eq:num:benchmark_dynamics}
    x^+
    \in
    \{Ax+w:A\in\mathscr A,\ w\in\mc W\},
\end{equation}
where
$\mathscr A = \operatorname{conv} \{\overline A_1,\overline A_2,
\overline A_3,\overline A_4\}.$
The vertex matrices are
\begin{equation}
\label{eq:num:benchmark_vertices}
    \begin{aligned}
        \overline A_1
        &=
        \begin{bmatrix}
            0.90 & 0.35\\
           -0.20 & 0.85
        \end{bmatrix},
        &
        \overline A_2
        &=
        \begin{bmatrix}
            0.85 & -0.30\\
            0.25 &  0.90
        \end{bmatrix},
        \\[-0.2mm]
        \overline A_3
        &=
        \begin{bmatrix}
            0.92 & 0.18\\
           -0.28 & 0.88
        \end{bmatrix},
        &
        \overline A_4
        &=
        \begin{bmatrix}
            0.87 & -0.15\\
            0.18 &  0.83
        \end{bmatrix}.
    \end{aligned}
\end{equation}
The state constraint and disturbance sets are
\begin{equation}
\label{eq:num:benchmark_sets}
    \mc X=[-1,1]^2,
    \qquad
    \mc W=[-0.025,0.025]^2.
\end{equation}
Since $\mc W$ is a centered box, its support function is
\[
    \support{\mc W}{y}
    =
    0.025\|y\|_1,
    \qquad y\in\R^2.
\]

\begin{figure}[htbp!]
    \centering  
    \includegraphics[width=0.9\columnwidth]{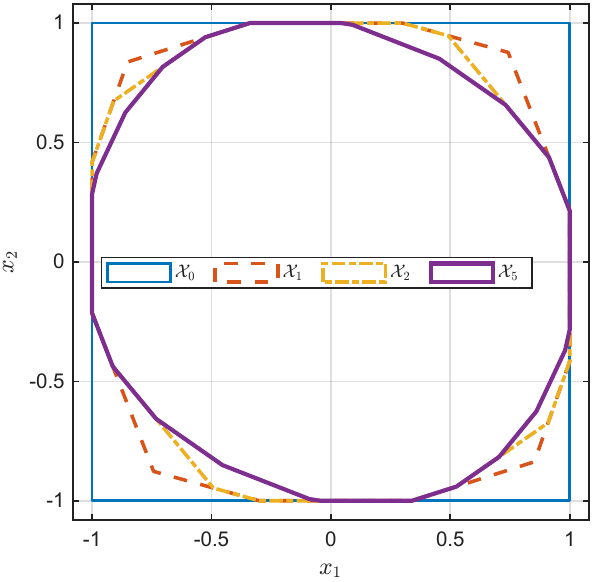}
    \setlength{\abovecaptionskip}{1pt}
    \caption{Selected iterates for the four-vertex polytopic difference
    inclusion, with
    \(\mc X_{5}=\mc M\).}
    \label{fig:num:benchmark_summary}
\end{figure}

Starting from $\mc X_0=\mc X$, the three implementations produced
the same polyhedral iterates up to the prescribed numerical tolerance.
Equality of the corresponding sets was verified at every iteration
through bidirectional set-inclusion tests implemented as linear
programs. The three exact stopping conditions \eqref{eq:finite:stationarity_condition}--\eqref{eq:finite:incremental_condition} were first satisfied at
the common determination index
$k^\star=5.$
Hence, all three implementations returned
$\mc M=\mc X_{5}.$
After complete redundancy removal, $\mc M$ is represented by
$24$ facet-defining inequalities and occupies
approximately $80.9\%$ of the area of the original constraint set
$\mc X$.

Figure~\ref{fig:num:benchmark_summary} shows the selected iterates
$\mc X_0$, $\mc X_1$, $\mc X_2$, and
$\mc X_{5}$. The early predecessor iterations
transform the axis-aligned box into a polytope with additional inclined
facets, while the subsequent iterates contract toward $\mc M$. As also expected from Corollary~\ref{cor:MRPI}, the central symmetry of $\mc X$ and $\mc W$ is preserved by all displayed iterates.

\subsection{Effect of Redundancy Removal}
\label{subsec:redundancy_timing}

For the four-vertex benchmark introduced in
Section~\ref{subsec:validation_three_iterations}, the standard,
self-restricted, and incremental implementations returned the same
finitely determined set
$\mc M=\mc X_{5}$. At the determination index,
the corresponding unreduced half-space representations of $\mc M$
contain $5460$, $12500$, and $5460$ inequalities, respectively.
Complete redundancy removal reduces each representation to an
equivalent irredundant representation with $24$ facet-defining
inequalities. The reduction in representation size is therefore
substantial for all three implementations, but its effect on
computation time depends on how the representation of the current
iterate enters subsequent updates.

To examine this tradeoff over different determination indices, consider
the six two-vertex polytopic matrix families
\begin{equation}
    \label{eq:num:redundancy_family}
    \begin{aligned}
        \mathscr A_j
        &=
        \operatorname{conv}
        \{\gamma_j\overline A_1,\gamma_j\overline A_2\},
        \qquad j=1,\ldots,6,
        \\
        (\gamma_1,\ldots,\gamma_6)
        &=
        (0.80,0.88,0.94,0.98,1.00,1.01).
    \end{aligned}
\end{equation}
Here, $\overline A_1$ and $\overline A_2$ are the first two
vertex matrices in \eqref{eq:num:benchmark_vertices}, and the state
constraint and additive disturbance sets are those in
\eqref{eq:num:benchmark_sets}. The scaling factors $\{\gamma_j\}_{j=1}^6$ 
were calibrated so that the
corresponding set iterations have determination indices
\[
k_j^*=j, \qquad j=1,\ldots,6.
\]
For comparison, the same six matrix families were also evaluated with
the disturbance set replaced by $\mc W=\{0\}$, while keeping the
state constraint set unchanged.

For every implementation and test case, enabling or disabling
redundancy removal produced the same finitely determined $\mc M$, up to 
numerical tolerance. 
For each configuration, each variant was first run once for warm-up.
Four paired timing runs were then performed, with the execution order
alternated between successive pairs.
Define the runtime ratio
\begin{equation}
\label{eq:num:redundancy_ratio}
    \eta
    :=
    \frac{t_{\mathrm{red}}}{t_{\mathrm{unred}}},
\end{equation}
where $t_{\mathrm{red}}$ and $t_{\mathrm{unred}}$ denote the
corresponding median computation times. Thus, $\eta<1$ indicates
that complete redundancy removal decreases the total computation time.

Table~\ref{tab:num:redundancy_all_methods} reports the detailed results
for $\mathscr A_6$ under the additive disturbance set in
\eqref{eq:num:benchmark_sets}. In the second and third columns, entries
are reported as unreduced/reduced. The second column gives the numbers
of inequalities in the corresponding half-space representations of
$\mc M$, while the third gives the median computation times. The
last column is reported as additive-disturbance/disturbance-free and
counts, among the six matrix families, how many cases satisfy
$\eta<1$.

\begin{table}[htbp!]
\centering
\caption{Effects of complete redundancy removal
on representation size and computation time
for $\mathscr A_6$ with additive disturbances,
together with the number of cases benefiting
from reduction across all six matrix families.}
\label{tab:num:redundancy_all_methods}
\small
\renewcommand{\arraystretch}{1.08}
\setlength{\tabcolsep}{2.4pt}
    \begin{tabular}{lcccc}
        \hline
        Implementation
        & \(\#\) inequalities
        & Time [s]
        & \(\eta\)
        & \shortstack{\(\#\) cases\\with \(\eta<1\)}
        \\
        \hline
        Standard & $508/20$ & $2.66/0.44$ & $0.167$ & $5/4$ \\
        Self-restricted & $2916/20$ & $22.05/0.43$ & $0.020$ & $5/4$ \\
        Incremental & $508/20$ & $4.05/5.32$ & $1.314$ & $0/0$ \\
        \hline
    \end{tabular}
    
\end{table}

Complete redundancy removal was faster in
$5$ of the six additive-disturbance
cases for the standard implementation and in
$5$ cases for the self-restricted
implementation. As the determination index increased from $1$ to
$6$, the endpoint runtime ratios $\eta$ changed from
$1.263$ to $0.167$ and from
$1.159$ to $0.020$, respectively.
The additional reduction cost was therefore not recovered in every
small case. In the disturbance-free tests, reduction was faster in
$4$ and $4$ of
the six cases, respectively.
For $k^\star=0$, both variants execute the same stopping test
without a reduction step, so differences in their measured times
reflect timing variation.
These results illustrate the tradeoff discussed in
Section~\ref{subsubsec:polyhedral_complexity}. The standard and
self-restricted implementations directly propagate the half-space
representation of $\mc X_k$ in the next predecessor update.
Reducing the number $m_k$ of retained inequalities therefore reduces
the number of transformed constraint normals and support-function
evaluations, as well as the sizes of the linear programs arising in
subsequent stopping and feasibility tests. This benefit is especially
pronounced for the self-restricted implementation, whose unreduced
representation grows as $m_{\mc X}(s+1)^k$.

The incremental implementation behaved differently. Complete
redundancy removal was slower in all twelve tested configurations.
Although it reduces the representation size of $\mc X_k$, it does
not reduce the number of candidate inequalities generated for
$\mc Y_{k+1}$. The direct construction of $\mc Y_{k+1}$ still
generates $m_{\mc X}s^{k+1}$ inequalities from the
$s^{k+1}$ matrix-index sequences, independently of the representation
size of $\mc X_k$. Complete reduction therefore introduces
additional redundancy-test linear programs without reducing this
constraint-generation cost.

Thus, complete reduction at every iteration favored the standard and
self-restricted implementations most strongly in the larger tested
cases, while it introduced additional overhead for the direct
incremental implementation. Periodic or
threshold-triggered reduction may nevertheless remain useful for
controlling the storage required by the latter.

\subsection{Computational Scaling}
\label{subsec:computational_scaling}
We examine the dependence of computation time on the state dimension
$n$, the number $s$ of vertex matrices, the number $m_{\mc X}$
of state-constraint inequalities, and the determination index
$k^\star$. The four experiments are summarized in
Table~\ref{tab:num:sweep_settings}. Redundancy removal and storage of
the full iterate history are disabled throughout.

\begin{table}[htbp!]
\caption{Settings for the computational scaling experiments.}
\label{tab:num:sweep_settings}
\centering
\small
\renewcommand{\arraystretch}{1.05}
\setlength{\tabcolsep}{2.8pt}
    \begin{tabular}{c c c c c c}
        \hline
        Sweep
        & $\qquad$Values$\qquad$
        & \(\ n\ \)
        & \(\ s\ \)
        & \(\ m_{\mc X}\ \)
        & Execution
        \\
        \hline
        \(n\)
        & \(2,4,6,8,10\)
        & --
        & \(2\)
        & \(20\)
        & \(3\) updates
        \\
        \(s\)
        & \(1,\ldots,6\)
        & \(3\)
        & --
        & \(10\)
        & \(3\) updates
        \\
        \(m_{\mc X}\)
        & \(6,8,\ldots,18\)
        & \(3\)
        & \(2\)
        & --
        & \(3\) updates
        \\
        \(k^\star\)
        & \(1,\ldots,6\)
        & \(2\)
        & \(2\)
        & \(4\)
        & exact stopping
        \\
        \hline
    \end{tabular}
\end{table}

The test problems are generated using the MATLAB
Mersenne--Twister random-number generator. For dimension $q$, the
vertex matrices and disturbance set are constructed as
\begin{equation}
\label{eq:num:scaling_generator}
    \begin{aligned}
        B_i^{(q)}
        &=
        I_q+0.35R_i^{(q)}+0.25J_q^+-0.15J_q^-,
        \\
        A_i^{(q)}
        &=
        0.84\,
        \frac{B_i^{(q)}}{\|B_i^{(q)}\|_2},
        \qquad
        \mc W_q=[-0.006,0.006]^q,
    \end{aligned}
\end{equation}
where the entries of $R_i^{(q)}$ are independently drawn from 
the standard normal distribution,
and $J_q^+$ and $J_q^-$ have ones on the first
superdiagonal and first subdiagonal, respectively, and zeros
elsewhere.
Before each sweep, the random-number generator is
reset using \texttt{rng(seed,`twister')}. The $n$-, $s$-,
and $m_{\mathcal X}$-sweeps use seeds 110, 206, and 318,
respectively. 
Constraint directions are drawn before vertex matrices.
The $n$-sweep uses a common ten-dimensional instance
with $H_{\mc X}=[D;-D]$, where
$D\in\R^{10\times10}$ consists of independent
standard Gaussian rows normalized to unit Euclidean
norm. For each $n$, leading $n\times n$ vertex
blocks are rescaled as in (5.6), while constraint
normals are truncated to $n$ entries and
renormalized. For the $s$- and $m_{\mc X}$-sweeps,
$H_{\mc X}$ starts with $[I_3;-I_3]$ and appends
normalized standard Gaussian rows in
positive-negative pairs. The $s$-sweep takes the
first $s$ of six common vertices. The
$m_{\mc X}$-sweep takes the first $m_{\mc X}$
of eighteen common constraints. All three sweeps
use $h_{\mc X}=\mathbf{1}_{m_{\mc X}}$.

In the first three experiments, emptiness tests are disabled and each
implementation performs three prescribed updates. The exact stopping
test is evaluated but does not terminate these runs. The
$k^\star$-sweep uses the calibrated family in
\eqref{eq:num:redundancy_family} and terminates when the corresponding
exact stopping condition is first satisfied to tolerance $10^{-7}$.
For each sweep, one complete pass over all problem--method
combinations is used for warm-up. The dimension sweep uses nine
interleaved timing passes, while the other sweeps use five.
\begin{figure}[htbp!]
    \centering
    \includegraphics[width=0.93\columnwidth]{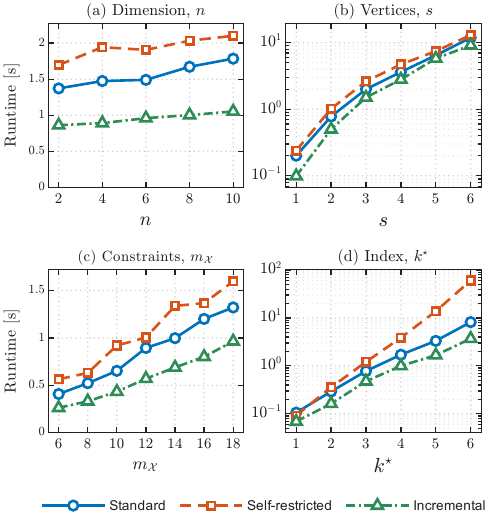}
    \setlength{\abovecaptionskip}{0pt}
    \caption{Median runtimes in the four scaling experiments.}
    \label{fig:num:runtime_scaling}
\end{figure}

Figure~\ref{fig:num:runtime_scaling} shows that the state dimension has
a comparatively modest effect over the tested range. The ratios of
the median times at $n=10$ to those at $n=2$ are approximately
$1.30$, $1.24$, and
$1.22$ for the standard, self-restricted, and
incremental implementations, respectively.
Increasing $m_{\mc X}$ from $6$ to $18$ raises the corresponding
times by factors of approximately $3.23$,
$2.83$, and $3.66$.
This growth is consistent with the increase in propagated constraint
normals and support-function evaluations.
The effects of $s$ and $k^\star$ are substantially stronger.
Increasing $s$ from $1$ to $6$ multiplies the median times of
the standard, self-restricted, and incremental implementations by
approximately $59.25$, $54.56$, and
$92.42$, respectively. Increasing
$k^\star$ from $1$ to $6$ multiplies the corresponding
end-to-end times by approximately $77.80$,
$661.53$, and $53.89$.
These trends agree with the operation counts and unreduced
constraint-growth expressions in
Section~\ref{subsubsec:polyhedral_complexity}. In particular, the
stronger dependence of the incremental implementation on $s$
reflects the enumeration of $s^{k+1}$ matrix-index sequences, while
the rapid growth of the self-restricted implementation with
$k^\star$ reflects its unreduced count
$m_{\mc X}(s+1)^k$.

The direct incremental implementation nevertheless remains the fastest over
the tested parameter ranges. Its stronger dependence on $s$ is
already visible in Fig.~\ref{fig:num:runtime_scaling}(b), where its
runtime has the largest endpoint growth factor and its relative
advantage narrows as $s$ increases.

\section{Conclusion}
\label{sec:conclusion}
This article has developed a unified theoretical and computational treatment of maximal robust positive invariance for linear difference inclusions. The proposed framework unifies the standard, self-restricted, and incremental recursions and their exact stopping tests. Theorems \ref{thm:nonempty} and \ref{thm:finite_strict_nonempty} establish an exact nonemptiness criterion and sufficient conditions for finite determination, respectively, with an explicit bound on the determination index provided by the latter. In the polyhedral setting, the exact implementations demonstrate that set-theoretic equivalence does not imply computational equivalence. The theoretical analysis and numerical results identify the matrix-family size, determination index, and redundancy-removal strategy as key factors in selecting an implementation.

The flexibility in choosing the initial pre-fixed outer set in Theorem~\ref{thm:pre-fixed-iteration} opens a natural direction for future work on tighter or problem-adapted initialization strategies that may reduce the computational burden of the descending iteration. A second direction is to exploit the structure of the generated constraints to develop more scalable implementations. Further work may also extend the analysis to dynamics with constrained switching or control inputs.

\section*{Acknowledgment}
During the preparation of this manuscript, the authors
used ChatGPT (GPT-5.6 Sol, OpenAI) for grammar checking, mathematical proof checking, and coding
assistance. The authors independently verified the
AI-assisted outputs and take full responsibility
for the content of this manuscript.

\section*{Appendix}
\label{sec:app}
\renewcommand\thesubsection{\Alph{subsection}}
\setcounter{subsection}{0}

\subsection{Proof of Theorem~\ref{thm:finite_determination}}
\label{app:finite_determination}
We first prove the equivalence between (i) and (ii). Suppose that
$\mc M$ is finitely determined. Then $\mc M=\mc X_k$ for some
$k\in\N$. Since $\mc M$ is a fixed point of $\mc F^{-1}$, we have 
$\mc X_{k+1}
    =
    \mc F^{-1}(\mc X_k)
    =
    \mc F^{-1}(\mc M)
    =
    \mc M
    =
    \mc X_k.$
Hence, \eqref{eq:finite:stationarity_condition} holds. Conversely, suppose that
$\mc X_k\subseteq\mc X_{k+1}$ for some $k\in\N$. Since the sequence
$\{\mc X_k\}_{k\in\N}$ is decreasing, we have 
$\mc X_{k+1}\subseteq\mc X_k,$
and therefore
$\mc X_k=\mc X_{k+1} = \mc F^{-1}(\mc X_k).$
Thus, $\mc X_k$ is a fixed point of $\mc F^{-1}$. Since $\mc M$ is
the greatest fixed point, it holds that
$ \mc X_k\subseteq\mc M.$
On the other hand, Corollary~\ref{cor:MRPI} gives
$\mc M\subseteq\mc X_k$. Hence,
$\mc M=\mc X_k,$
so $\mc M$ is finitely determined.

We next prove the equivalence between (ii) and (iii). By Proposition~\ref{prop:equivalent_iterations}, irrespective of which
of the three equivalent schemes is used to generate the sequence, one
has
$ \mc X_{k+1} = \overline{\mc X}_{k+1}\bigcap\mc X.$
Since $\mc X_k\subseteq\mc X$, it follows that
\[
    \mc X_k\subseteq\mc X_{k+1}
    \quad\Longleftrightarrow\quad
    \mc X_k\subseteq\overline{\mc X}_{k+1}.
\]
Thus, (ii) and (iii) are equivalent.

Finally, by the incremental iteration \eqref{eq:equiv_iteration_3}, we have
\[
    \mc X_k\subseteq\mc X_{k+1}
    \ \Longleftrightarrow\ 
    \mc X_k
    \subseteq
    \mc X_k\bigcap\mc Y_{k+1}
    \ \Longleftrightarrow\ 
    \mc X_k\subseteq\mc Y_{k+1}.
\]
Therefore, (ii) and (iv) are equivalent.

If any one of
\eqref{eq:finite:stationarity_condition},
\eqref{eq:finite:stationarity_condition_}, or
\eqref{eq:finite:incremental_condition}
holds for some $k\in\N$, then all three conditions hold for the same
$k$. In particular,
$\mc X_k=\mc X_{k+1}.$
By the equivalence between (i) and (ii), we have
$\mc M=\mc X_k=\mc X_{k+1}.$
Moreover, Corollary~\ref{cor:MRPI} gives
$\mc M=\mc X_\infty$. Therefore, it holds that
$\mc M = \mc X_\infty = \mc X_k = \mc X_{k+1}.$
\hfill$\square$

\subsection{Proof of Proposition~\ref{prop:minimal_rpi}}
\label{app:minimal_rpi}
By uniform exponential stability, there exist constants $c\geq1$ and
$\rho\in[0,1)$ such that
\[
    \|\Psi_{k,i}\|
    \leq
    c\rho^{k-i},
    \qquad
    0\leq i\leq k,
\]
for every
$\boldsymbol{\Psi}_k\in\mathfrak B_k(\mathscr A)$ and every $k\geq1$.
Since $\mc W$ is bounded, there exists $\bar w\geq0$ such that
$\|w\|\leq\bar w$ for every $w\in\mc W$. Hence, every
$x\in\mc R_k$ satisfies
\[
    \|x\|
    \leq
    \sum_{i=0}^{k-1}
    \|\Psi_{k,i+1}w_i\|
    \leq
    c\bar w
    \sum_{i=0}^{k-1}\rho^{k-i-1}
    \leq
    \frac{c\bar w}{1-\rho}.
\]
Thus, the sets $\{\mc R_k\}_{k\geq1}$ are uniformly bounded.
For each $\ell\geq1$, define
$\mc K_\ell := \operatorname{cl} ( \bigcup_{k\geq\ell}\mc R_k ).$
Since $\mathscr A$ and $\mc W$ are nonempty, every $\mc R_k$ is
nonempty. Therefore, each $\mc K_\ell$ is nonempty. Moreover,
$\mc K_\ell$ is closed and bounded, and hence compact. The sequence is
decreasing:
$ \mc K_{\ell+1}\subseteq\mc K_\ell, \ \ell\geq1.$
Consequently, by the Cantor intersection theorem,
$\mc R = \bigcap_{\ell\geq1}\mc K_\ell$
is nonempty. Since $\mc R$ is a closed subset of the compact set
$\mc K_1$, it is also compact.

We next prove robust positive invariance of $\mc R$. Fix an arbitrary $x\in\mc R$. 
Since
$x\in\mc K_j$ for every $j\geq1$, there exist integers $k_j\geq j$
and points $x_j\in\mc R_{k_j}$ such that
$\|x_j-x\|\leq 1/j.$
Here, $k_j\geq j$ ensures that the selected reachable-set indices
become arbitrarily large, while
$\|x_j-x\|\leq1/j$ implies that $x_j\to x$. Fix arbitrary
$A\in\mathscr A$ and $w\in\mc W$. Appending $A$ and $w$ to the
admissible sequences generating $x_j$ gives
$Ax_j+w\in\mc R_{k_j+1}.$
For any $\ell\geq1$, one has $k_j+1\geq\ell$ for all sufficiently
large $j$, and therefore
$Ax_j+w\in\mc K_\ell.$
Since $\mc K_\ell$ is closed and $Ax_j+w\to Ax+w$, it follows that
$Ax+w\in\mc K_\ell$. As $\ell$ was arbitrary, it holds that 
$Ax+w\in\bigcap_{\ell\geq1}\mc K_\ell=\mc R.$
Thus, we have $\mc F(\mc R)\subseteq\mc R.$

Finally, let $\mc S\subseteq\R^n$ be any nonempty closed set satisfying
$\mc F(\mc S)\subseteq\mc S$. Choose $\bar x\in\mc S$, and let
$x\in\mc R$. Select integers $k_j\to\infty$ and points
$x_j\in\mc R_{k_j}$ such that $x_j\to x$. For every $j\geq1$, by the
definition of $\mc R_{k_j}$, there exist
$\boldsymbol{\Psi}_{k_j} \in \mathfrak B_{k_j}(\mathscr A)$
and disturbances
$w_{j,0},\ldots,w_{j,k_j-1}\in\mc W$ such that
$ x_j = \sum_{i=0}^{k_j-1} \Psi_{k_j,i+1}w_{j,i}.$
Applying the same admissible matrix and disturbance sequences from
$\bar x$ and using $\mc F(\mc S)\subseteq\mc S$ yields
$\Psi_{k_j,0}\bar x+x_j\in\mc S.$
Moreover, uniform exponential stability gives
$\|\Psi_{k_j,0}\bar x\| \leq c\rho^{k_j}\|\bar x\| \rightarrow0.$
Since $x_j\to x$, it follows that
$\Psi_{k_j,0}\bar x+x_j\to x.$
As $\mc S$ is closed, one has $x\in\mc S$. Since $x\in\mc R$ was arbitrary, 
it follows that $\mc R\subseteq\mc S.$
\hfill$\square$

\subsection{Proof of Theorem~\ref{thm:finite_strict_nonempty}}
\label{app:finite_strict_nonempty}
By \eqref{eq:finite_invariant_margin} and $0\in\mathbb B$,
we have
$\mc F(\mc S)\subseteq\mc S\subseteq\mc X$.
Thus, $\mc S$ is a nonempty robust positively invariant
set, and hence
$\mc S\subseteq\mc M\subseteq\mc X_k$
for every $k\in\N$.
In particular, $\mc X_{\bar k}$ is nonempty.
Since $\mc X_{\bar k}$ is bounded and $\mc S$ is nonempty,
the quantity $\beta$ in \eqref{eq:finite_beta} satisfies
$0\leq\beta<\infty$.

Fix any $q\geq1$, any $x_0\in\mc X_{\bar k}$,
and arbitrary admissible matrix and disturbance sequences
$A_0,\ldots,A_{q-1}\in\mathscr A$ and
$w_0,\ldots,w_{q-1}\in\mc W$.
Since $\mc S$ is nonempty and closed, there exists
$\bar x_0\in\mc S$ such that
$\|x_0-\bar x_0\|=\operatorname{dist}(x_0,\mc S).$
Consider the trajectories generated from $x_0$ and
$\bar x_0$ under the same matrix and disturbance sequences:
\[
    x_{j+1}=A_jx_j+w_j,
    \qquad
    \bar x_{j+1}=A_j\bar x_j+w_j,
\]
for $j=0,\ldots,q-1$.
Since $\mc F(\mc S)\subseteq\mc S$, we have
$\bar x_j\in\mc S$ for every $j=0,\ldots,q$.
Since $\bar x_q\in\mc S$, uniform exponential stability
and the use of identical matrix and disturbance
sequences give
\begin{align*}
    \operatorname{dist}(x_q,\mc S) \leq\|x_q-\bar x_q\|=\|\Psi_{q,0}(x_0-\bar x_0)\|\leq c\rho^q\beta.
\end{align*}
Now let $q\geq\bar k+1$ satisfy $c\rho^q\beta\leq\varepsilon$.
The preceding estimate yields
$\|x_q-\bar x_q\|\leq\varepsilon$, and hence we have
\[
    x_q\in\mc S+\varepsilon\mathbb B
    \subseteq\mc X.
\]
Thus, every admissible $q$-step trajectory starting
from $\mc X_{\bar k}$ reaches $\mc X$, which implies
$\mc X_{\bar k}\subseteq\mc Y_q$.
Since $q-1\geq\bar k$ and the sequence
$\{\mc X_k\}_{k\in\N}$ is decreasing, we obtain
\[
    \mc X_{q-1}
    \subseteq\mc X_{\bar k}
    \subseteq\mc Y_q.
\]
By the exact stopping condition
\eqref{eq:finite:incremental_condition},
Theorem~\ref{thm:finite_determination} yields
$\mc M=\mc X_{q-1}$.

Finally, when $\rho\in(0,1)$ and $\beta>0$,
the condition $c\rho^q\beta\leq\varepsilon$
is equivalent to
\[
    q\geq
    \frac{\ln(c\beta/\varepsilon)}
         {\ln(1/\rho)}.
\]
Together with $q\geq\bar k+1$, this gives
\eqref{eq:finite_explicit_bound}.
If $\rho=0$ or $\beta=0$, the condition holds
for every $q\geq1$, so $q=\bar k+1$ is admissible.
This completes the proof.
\hfill$\square$

\subsection{Proof of Corollary~\ref{cor:polyhedral_iterates}}
\label{app:polyhedral_iterates}
We first consider the finite-family case
$\mathscr A=\mathscr A_{\rm v}$. Since $\mc X_0=\mc X$ is
polyhedral, suppose inductively that $\mc X_k$ is polyhedral.
By~\eqref{eq:poly:pontryagin_difference}, the boundedness of
$\mc W$ ensures that $\mc X_k-\mc W$ admits a finite
half-space representation and is therefore polyhedral. Hence,
each preimage
$\overline A_i^{-1}(\mc X_k-\mc W),\ i=1,\ldots,s,$
is polyhedral, as is their finite intersection. Intersecting
this set with either $\mc X$ or $\mc X_k$ shows that
$\mc X_{k+1}$ is polyhedral. By induction, $\mc X_k$ is
polyhedral for every $k\in\N$.

We next consider the sets $\mc Y_k$. Clearly,
$\mc Y_0=\mc X$ is polyhedral. For each $k\geq1$ and each
$\boldsymbol{\Psi}_k\in\mathfrak B_k(\mathscr A_{\rm v})$,
define the accumulated disturbance set
$\mc D_{\boldsymbol{\Psi}_k} := \sum_{i=0}^{k-1}\Psi_{k,i+1}\mc W.$
Since $\mc W$ is nonempty and bounded,
$\mc D_{\boldsymbol{\Psi}_k}$ is also nonempty and bounded,
and its support function is therefore finite-valued. It follows
from~\eqref{eq:poly:X} that
\[
    \mc X-\mc D_{\boldsymbol{\Psi}_k}
    =
    \left\{
        z\in\R^n:
        H_{\mc X}z
        \leq
        h_{\mc X}
        -
        \support{\mc D_{\boldsymbol{\Psi}_k}}{H_{\mc X}}
    \right\},
\]
which is polyhedral. Its preimage under $\Psi_{k,0}$ is
therefore polyhedral. Since
$\mathfrak B_k(\mathscr A_{\rm v})$ is finite,
\eqref{eq:comp:Y_k} expresses $\mc Y_k$ as a finite
intersection of polyhedral sets. Thus, $\mc Y_k$ is
polyhedral for every $k\in\N$.

Now consider the polytopic case
$\mathscr A=\operatorname{conv}(\mathscr A_{\rm v}).$
Since $\mc X$ is polyhedral, it is convex, and
Corollary~\ref{cor:MRPI} implies that $\mc X_k$ is convex for every
$k\in\N$. Hence,
$\mc X_k-\mc W$
is also convex.
Because $\mathscr A_{\rm v}\subseteq\mathscr A$, one has
$\bigcap_{A\in\mathscr A} A^{-1}(\mc X_k-\mc W) \subseteq
 \bigcap_{i=1}^{s} \overline A_i^{-1}(\mc X_k-\mc W).$
For the reverse inclusion, let
$x \in \bigcap_{i=1}^{s} \overline A_i^{-1}(\mc X_k-\mc W).$
Then $\overline A_i x\in\mc X_k-\mc W$ for every $i=1,\ldots,s$.
For any $A\in\mathscr A$, write
\[
    A=\sum_{i=1}^{s}\lambda_i\overline A_i,
    \qquad
    \lambda_i\geq0,
    \qquad
    \sum_{i=1}^{s}\lambda_i=1.
\]
By convexity of $\mc X_k-\mc W$, it holds that
$ Ax = \sum_{i=1}^{s}\lambda_i\overline A_i x \in \mc X_k-\mc W.$
Thus,
$x\in A^{-1}(\mc X_k-\mc W)$ for every $A\in\mathscr A$, which
proves~\eqref{eq:poly:vertex_predecessor}.
We next prove the vertex reduction for $\mc Y_k$. Fix
$k\geq1$, an initial state $x_0$, a disturbance sequence
$w_0,\ldots,w_{k-1}\in\mc W$, and a matrix sequence
$A_0,\ldots,A_{k-1}\in\mathscr A$. For each
$j=0,\ldots,k-1$, write
\[
    A_j
    =
    \sum_{r=1}^{s}
    \lambda_{j,r}\overline A_r,
    \qquad
    \lambda_{j,r}\geq0,
    \qquad
    \sum_{r=1}^{s}\lambda_{j,r}=1.
\]
Expanding the state recursion gives
\[
    x_k
    =
    \sum_{\sigma\in\{1,\ldots,s\}^{k}}
    \alpha_\sigma x_k^\sigma,
    \qquad
    \alpha_\sigma
    :=
    \prod_{j=0}^{k-1}\lambda_{j,\sigma_j},
\]
where $x_k^\sigma$ is the state reached from the same $x_0$ and under
the same disturbance sequence, but with the vertex-matrix sequence
$\overline A_{\sigma_0},\ldots,
    \overline A_{\sigma_{k-1}}.$
The coefficients $\alpha_\sigma$ are nonnegative and satisfy
$\sum_{\sigma\in\{1,\ldots,s\}^{k}}\alpha_\sigma=1.$
Therefore, if $x_k^\sigma\in\mc X$ for every vertex sequence
$\sigma$, then $x_k\in\mc X$ by convexity of $\mc X$. The converse is
immediate because every vertex sequence is admissible for
$\mathscr A$. Thus, requiring the state constraint at time $k$ for all
matrix sequences in $\mathscr A$ is equivalent to requiring it for all
vertex-matrix sequences. This proves~\eqref{eq:poly:vertex_Yk}.

Equations~\eqref{eq:poly:vertex_predecessor}
and~\eqref{eq:poly:vertex_Yk} reduce the polytopic case to the
finite-family case. Hence, $\mc X_k$ and $\mc Y_k$ are polyhedral for
every $k\in\N$. Finally, if $\mc M$ is finitely determined, then
$\mc M=\mc X_k$ for some $k\in\N$, and is therefore polyhedral.
\hfill$\square$

\bibliographystyle{elsarticle-num}
\bibliography{references}

\begin{thebibliography}{10}
\expandafter\ifx\csname url\endcsname\relax
  \def\url#1{\texttt{#1}}\fi
\expandafter\ifx\csname urlprefix\endcsname\relax\def\urlprefix{URL }\fi
\expandafter\ifx\csname href\endcsname\relax
  \def\href#1#2{#2} \def\path#1{#1}\fi

\bibitem{Khalil2002NonlinearSystems}
H.~K. Khalil, Nonlinear Systems, 3rd Edition, Prentice Hall, 2002.

\bibitem{Blanchini1999SetInvariance}
F.~Blanchini, Set invariance in control, Automatica 35~(11) (1999) 1747--1767.

\bibitem{BlanchiniMiani2015}
F.~Blanchini, S.~Miani, Set-Theoretic Methods in Control, 2nd Edition, Systems \& Control: Foundations \& Applications, Birkh{\"a}user, 2015.

\bibitem{Aubin1991ViabilityTheory}
J.-P. Aubin, Viability Theory, Systems \& Control: Foundations \& Applications, Birkh{\"a}user, 1991.

\bibitem{BertsekasRhodes1971MinimaxReachability}
D.~P. Bertsekas, I.~B. Rhodes, On the minimax reachability of target sets and target tubes, Automatica 7~(2) (1971) 233--247.

\bibitem{MitchellBayenTomlin2005HJReachability}
I.~M. Mitchell, A.~M. Bayen, C.~J. Tomlin, A time-dependent {Hamilton--Jacobi} formulation of reachable sets for continuous dynamic games, IEEE Transactions on Automatic Control 50~(7) (2005) 947--957.

\bibitem{KurzhanskiyVaraiya2007EllipsoidalReachability}
A.~A. Kurzhanskiy, P.~Varaiya, Ellipsoidal techniques for reachability analysis of discrete-time linear systems, IEEE Transactions on Automatic Control 52~(1) (2007) 26--38.

\bibitem{Tabuada2009VerificationHybridSystems}
P.~Tabuada, Verification and Control of Hybrid Systems: A Symbolic Approach, Springer, 2009.

\bibitem{BeltaYordanovGol2017FormalMethods}
C.~Belta, B.~Yordanov, E.~A. Gol, Formal Methods for Discrete-Time Dynamical Systems, Springer, 2017.

\bibitem{Bertsekas1972InfiniteReachability}
D.~P. Bertsekas, Infinite-time reachability of state-space regions by using feedback control, IEEE Transactions on Automatic Control 17~(5) (1972) 604--613.

\bibitem{GutmanCwikel1987MaximalConstraints}
P.-O. Gutman, M.~Cwikel, An algorithm to find maximal state constraint sets for discrete-time linear dynamical systems with bounded controls and states, IEEE Transactions on Automatic Control 32~(3) (1987) 251--254.

\bibitem{GilbertTan1991MaximalOutputAdmissible}
E.~G. Gilbert, K.~T. Tan, Linear systems with state and control constraints: The theory and application of maximal output admissible sets, IEEE Transactions on Automatic Control 36~(9) (1991) 1008--1020.

\bibitem{KolmanovskyGilbert1998DisturbanceInvariant}
I.~Kolmanovsky, E.~G. Gilbert, Theory and computation of disturbance invariant sets for discrete-time linear systems, Mathematical Problems in Engineering 4~(4) (1998) 317--367.

\bibitem{DoreaHennet1999ABInvariant}
C.~E.~T. D{\'o}rea, J.-C. Hennet, {(A,B)}-invariant polyhedral sets of linear discrete-time systems, Journal of Optimization Theory and Applications 103~(3) (1999) 521--542.

\bibitem{LinAntsaklis2002RobustControlledInvariant}
H.~Lin, P.~J. Antsaklis, Robust controlled invariant sets for a class of uncertain hybrid systems, in: Proceedings of the 41st IEEE Conference on Decision and Control, Vol.~3, Las Vegas, NV, USA, 2002, pp. 3180--3181.

\bibitem{Shang2004MaximalRobustControlledInvariant}
Y.~Shang, The maximal robust controlled invariant set of uncertain switched systems, in: Proceedings of American Control Conference, Vol.~6, Boston, MA, USA, 2004, pp. 5195--5196.

\bibitem{PluymersEtAl2005PolyhedralInvariant}
B.~Pluymers, J.~A. Rossiter, J.~A.~K. Suykens, B.~{De Moor}, The efficient computation of polyhedral invariant sets for linear systems with polytopic uncertainty, in: Proceedings of American Control Conference, 2005, pp. 804--809.

\bibitem{HirataOhta2008ExactMOAS}
K.~Hirata, Y.~Ohta, Exact determinations of the maximal output admissible set for a class of nonlinear systems, Automatica 44~(2) (2008) 526--533.

\bibitem{RakovicFiacchini2008MaximalApproximation}
S.~V. Rakovi{\'c}, M.~Fiacchini, Invariant approximations of the maximal invariant set or ``encircling the square'', IFAC Proceedings Volumes 41~(2) (2008) 6377--6382.

\bibitem{AthanasopoulosBitsoris2009MaximalControlledInvariant}
N.~Athanasopoulos, G.~Bitsoris, A novel approach to the computation of the maximal controlled invariant set for constrained linear systems, in: European Control Conference (ECC), Budapest, Hungary, 2009, pp. 3124--3129.

\bibitem{RunggerTabuada2017RobustControlledInvariant}
M.~Rungger, P.~Tabuada, Computing robust controlled invariant sets of linear systems, IEEE Transactions on Automatic Control 62~(7) (2017) 3665--3670.

\bibitem{EsterhuizenAschenbruckStreif2020MaximalRPI}
W.~Esterhuizen, T.~Aschenbruck, S.~Streif, On maximal robust positively invariant sets in constrained nonlinear systems, Automatica 119 (2020) 109044.

\bibitem{WangJungersOng2021MaximalInvariant}
Z.~Wang, R.~M. Jungers, C.-J. Ong, Computation of the maximal invariant set of discrete-time linear systems subject to a class of non-convex constraints, Automatica 125 (2021) 109463.

\bibitem{RakovicZhang2023ImplicitMPI}
S.~V. Rakovi{\'c}, S.~Zhang, The implicit maximal positively invariant set, IEEE Transactions on Automatic Control 68~(8) (2023) 4738--4753.

\bibitem{Wang2023computation}
Z.~Wang, R.~Jungers, C.~J. Ong, Computation of invariant sets via immersion for discrete-time nonlinear systems, Automatica 147 (2023) 110686.

\bibitem{Ossareh2024complexity}
H.~R. Ossareh, I.~Kolmanovsky, On complexity bounds for the maximal admissible set of linear time-invariant systems, IEEE Transactions on Automatic Control 69~(9) (2024) 6389--6396.

\bibitem{RakovicVillanueva2017PolynomialMPI}
S.~V. Rakovi{\'c}, M.~E. Villanueva, The maximal positively invariant set: Polynomial setting, arXiv preprint arXiv:1712.01150 (2017).

\bibitem{BenlaoukliEtAl2009PWA}
H.~Benlaoukli, M.~Hovd, S.~Olaru, P.~Boucher, On the construction of invariant sets for piecewise affine systems using the transition graph, in: Proceedings of the 2009 IEEE International Conference on Control and Automation, Christchurch, New Zealand, 2009, pp. 122--127.

\bibitem{Janak2022ExactInvariant}
D.~Janak, P.-L. Garoche, B.~A{\c c}{\i}kme{\c s}e, Exact computation of maximal invariant sets for safe {Markov} chains---lattice theoretic approach, IEEE Transactions on Automatic Control 67~(12) (2022) 6980--6986.

\bibitem{AthanasopoulosSmpoukisJungers2017InvariantSets}
N.~Athanasopoulos, K.~Smpoukis, R.~M. Jungers, Invariant sets analysis for constrained switching systems, IEEE Control Systems Letters 1~(2) (2017) 256--261.

\bibitem{AthanasopoulosJungers2018Hybrid}
N.~Athanasopoulos, R.~M. Jungers, Combinatorial methods for invariance and safety of hybrid systems, Automatica 98 (2018) 130--140.

\bibitem{RakovicGielen2014InvariantFamilies}
S.~V. Rakovi{\'c}, R.~H. Gielen, Positively invariant families of sets for interconnected and time-delay discrete-time systems, SIAM Journal on Control and Optimization 52~(4) (2014) 2261--2283.

\bibitem{DeyBhasin2024Computation}
A.~Dey, S.~Bhasin, Computation of maximal admissible robust positive invariant sets for linear systems with parametric and additive uncertainties, IEEE Control Systems Letters 8 (2024) 1775--1780.

\bibitem{BorrelliBemporadMorari2017PredictiveControl}
F.~Borrelli, A.~Bemporad, M.~Morari, Predictive Control for Linear and Hybrid Systems, Cambridge University Press, 2017.

\bibitem{Birkhoff1967LatticeTheory}
G.~Birkhoff, Lattice Theory, 3rd Edition, American Mathematical Society, 1967.

\bibitem{DaveyPriestley2002}
B.~A. Davey, H.~A. Priestley, Introduction to Lattices and Order, 2nd Edition, Cambridge University Press, 2002.

\bibitem{Tarski1955LatticeFixpoint}
A.~Tarski, A lattice-theoretical fixpoint theorem and its applications, Pacific Journal of Mathematics 5~(2) (1955) 285--309.

\bibitem{Cousot1979}
P.~Cousot, R.~Cousot, Constructive versions of {T}arski's fixed point theorems, Pacific Journal of Mathematics 82~(1) (1979) 43--57.

\bibitem{StoltenbergHansen1994}
V.~Stoltenberg-Hansen, I.~Lindstr{\"o}m, E.~R. Griffor, Mathematical Theory of Domains, Cambridge University Press, 1994.

\bibitem{Shorten2007Stability}
R.~Shorten, F.~Wirth, O.~Mason, K.~Wulff, C.~King, Stability criteria for switched and hybrid systems, SIAM Review 49~(4) (2007) 545--592.

\bibitem{Wirth2002JSR}
F.~Wirth, The generalized spectral radius and extremal norms, Linear Algebra and its Applications 342 (2002) 17--40.

\bibitem{KouramasEtAl2005MinimalRPI}
K.~I. Kouramas, S.~V. Rakovi{\'c}, E.~C. Kerrigan, J.~C. Allwright, D.~Q. Mayne, On the minimal robust positively invariant set for linear difference inclusions, in: Proceedings of the 44th IEEE Conference on Decision and Control and the European Control Conference, Seville, Spain, 2005, pp. 2296--2301.

\bibitem{Rakovic2004}
S.~V. Rakovi{\'c}, E.~C. Kerrigan, K.~I. Kouramas, D.~Q. Mayne, Invariant approximations of robustly positively invariant sets for constrained linear discrete-time systems subject to bounded disturbances, Tech. Rep. CUED/F-INFENG/TR.473, Department of Engineering, University of Cambridge (2004).

\bibitem{Nguyen2011}
H.~N. Nguyen, S.~Olaru, F.~Stoican, On maximal robustly positively invariant sets, in: Proceedings of the 8th International Conference on Informatics in Control, Automation and Robotics (ICINCO), 2011, pp. 300--305.

\bibitem{Gurobi}
{Gurobi Optimization, LLC}, \href{https://www.gurobi.com}{{Gurobi Optimizer Reference Manual}} (2026).
\newline\urlprefix\url{https://www.gurobi.com}

\end{thebibliography}

\end{document}